\documentclass[12pt]{amsart}

\usepackage[english]{babel}

\usepackage[a4paper]{geometry}

\usepackage{amsmath,amssymb,amsfonts,amsthm,a4,fleqn,bbm,dsfont,enumerate}
\usepackage{enumitem,cite}
\newtheorem{definition}{Definition}[section] 
\newtheorem{lemma}[definition]{Lemma}
\newtheorem{example}[definition]{Example}
\newtheorem{theorem}[definition]{Theorem}
\newtheorem{corollary}[definition]{Corollary}
\newtheorem{remark}[definition]{Remark}
\newtheorem{conjecture}[definition]{Conjecture}

\usepackage{graphicx}
\graphicspath{C:\Users\akhil\OneDrive\Desktop\Library Mathematics\Other}
\usepackage[colorlinks=true, allcolors=blue]{hyperref}
\makeatletter
\@namedef{subjclassname@2020}{
	\textup{2020} Mathematics Subject Classification}
\makeatother
\begin{document}
\title{Center and Radius of a subset of metric space}
\author[A. Badra]{Akhilesh Badra}
\address{AKHILESH BADRA; Department of Mathematics, University of Delhi,
	Delhi--110 007, India}
\email{akhileshbadra028@gmail.com}
\author[H. K. Singh]{Hemant kumar Singh}
\address{HEMANT KUMAR SINGH; Department of Mathematics, University of Delhi,
	Delhi--110 007, India}
\email{hemantksingh@maths.du.ac.in}
\thanks{The first author is supported by research grant from the Council of Scientific and Industrial Research (CSIR), Ministry of Science and Technology, Government of India with reference number: 09/0045(13774)/2022-EMR-I}
	% with reference number- 09/0045(13774)/2022-EMR-I
%\author]{Hemant Kumar Singh}
%\affil[1]{Department of Mathematics, University of Delhi, Delhi 110007, India}

\date{}

%\title{Center and Radius of a Subset of Metric Space}
%\author{Akhilesh Badra\thanks{First author is supported by Council of Scientific and Industrial Research (Ministry of Science and
		%	Technology, Government of India) with reference number- 09/0045(13774)/2022-EMR-I} \hspace{0.4mm} and Hemant kumar Singh}
\maketitle

\begin{abstract}In this paper,
we introduce a notion of the center and radius of a subset $A$ of metric space $X$. In the Euclidean spaces, this notion can be seen as the extension of the center and radius of open/closed balls. %We have extended the notion of center and radius of balls in a metric space to the center and radius of a subset $A$ of a metric space $X.$
 The center and radius of a finite product of subsets of metric spaces, and a finite union of subsets of a metric space are also determined.\par 
 %An open ball of $X$ with the largest radius entirely contained in $A$ is also determined.
 For any  subset $A$ of metric space $X,$ there is a natural question to identify the open balls of  $X$ with the largest radius that are entirely contained in $A$. To answer this question, we introduce a notion of quasi-center and quasi-radius of a subset $A$ of  metric space $X$. We prove that the center of the largest open balls contained in $A$ belongs to the quasi-center of $A$, and its radius is equal to the quasi-radius of $A$. In particular, for the Euclidean spaces, we see that the center of largest open balls contained in $A$ belongs to the center of $A$, and its radius is equal to the radius of $A$.
\end{abstract}

\section{\textbf{Introduction}}
To extend classical geometric ideas beyond the Euclidean spaces, Maurice Fréchet introduced a notion of distance to more abstract setting in 1906. In his doctoral dissertation, he defined the distance between any two points within a given set \cite{james}.  %In 1906, in his doctoral dissertation he gave a notion of distance
% considered a distance function $d: X \times X \rightarrow \mathbb{R}$ between any two points $x$ and $y$ of a given set $X$. %such that it satisfies four postulates: $ d(x, y) \geq  0,$ $d(x, y) = 0 \iff x = y$, $d(x, y) = d(y, x),$ and $d(x, y) \leq d(x, z) + d(z, y),\forall$ $ x,y,z\in X.$
 A set along with this distance notion is called metric space, and this name was given by Hausdorff in 1914.
%Hausdorff gave it the name of metric spaces in 1914.
%The extension of these classical geometric ideas to more abstract settings began in earnest during the 19th and 20th centuries, as mathematicians sought to generalize the notion of distance and spatial relationships beyond the Euclidean framework. The development of metric spaces, introduced by Maurice Fréchet in 1906, provided a versatile and robust structure for analyzing a wide variety of spaces, encompassing not only geometric spaces but also function spaces and other abstract constructs.
In a metric space, the concept of distance is axiomatized, and allowing us for the study of convergence, continuity, compactness etc. which are essential in many areas of pure and  applied mathematics.\par
Understanding the geometric properties of subsets within metric spaces is a fundamental aspect of mathematical analysis and topology. The notion of open/closed balls provide essential tools for analyzing properties of metric spaces. Every ball is defined by its center and radius. The center of ball is in the middle of ball, from which the radius is measured. The purpose of the center of a ball is to act as the fixed reference from which all other points on the boundary of the ball are equidistant. It is natural to question whether this concept can be extended to any arbitrary subset of a metric space. Are there any points whithin a subset of a metric space that are equidistant from its boundary $?$ However, extending this notion of center and radius to a subset of metric spaces has significant challenges and opens new avenues for exploration.\par
%Metric space is an abstract set with a distance function, called a metric, that specifies a nonnegative distance between any two of its points in such a way that the following properties hold: (1) the distance from the first two the second equals zero if and only if the points are the same, (2) the distance from the first point to the second equals the distance from the second point to the first, and (3) the sum of the distance from the first point to the second and the distance from the second to a third exceeds or equals the distance from the first to the third. 

 This paper addresses these questions by introducing a novel framework for defining the center and radius of subsets within arbitrary metric spaces, 
%In this work, we define the center and radius of a subset $A$ of metric space $X$
 in a way that parallels to the idea of center and radius of balls in the Euclidean spaces. This definition not only preserves the intuitive geometric interpretation but also extends it to abstract settings. We further explore the properties of the center and radius in the context of finite products of subsets of metric spaces and a finite union of subsets within a metric space. \par
A central question addressed in this paper is the identification of the largest open balls that are entirely contained within a given subset $A$ of metric space 
$X$. To tackle this, we introduce a notion of quasi-center and quasi-radius. These concepts serve as pivotal tools in characterizing the maximal open balls contained in $A$, and we demonstrate that the center of these balls is located within the quasi-center of 
$A$, with their radius equating to the quasi-radius of 
$A$. Notably, when applied to the Euclidean spaces, our results show that the center and radius of the largest open balls contained in $A$ correspond directly to the center and radius of $A$ itself.\par
This paper is structured as follows: Section 2 consists of notations, terminology and basics of metric spaces that are used in this paper. Section 3 provides a detailed definition of the center and radius of a subset within a metric space, along with key properties and examples. Section 4 and section 5 extends these concepts to finite products and finite unions of subsets of metric space, respectively. In Section 6, we introduce the notions of quasi-center and quasi-radius, leading to a key result that connects these concepts to largest open balls contained within a subset. Finally, Section 7 discusses a conjecture about a relationship between the rank of different dimensional homology groups of a subset $A$ of metric space $X$, and outlines potential directions for future research.
\section{\textbf{Notations, Terminology and Basics}}
Let $(X,d_{X})$ be a metric space. The distance between two subsets $A$ and $B$ of $X$ is $d_{X}(A,B)=\inf\{d_{X}(a,b)|a\in A, b\in B\}.$ In many results of this paper, a subset is required to be nonclopen, which means that the subset is not both open and closed. The interior, closure and boundary of $A$ are denoted by $A^\circ, \overline{A}$ and $\partial_{X}(A)$, respectively. For details about metric spaces, we refer \cite{Robert}.\par
Now, we recall some basic results and properties of metric spaces.
\begin{itemize}
	\item Let $f:X\longrightarrow Y$ be an isometry between two metric spaces $X$ and $Y.$ For a subset $A$ of $X$, $f(\partial_{X}(A))=\partial_{Y}(f(A)).$ 
	\item For a subset $A$ of  metric space $X$,  $\partial_{X}(A^\circ)\subseteq\partial_{X}(A)$ and  $\partial_{X}(\overline{A})\subseteq\partial_{X}(A).$ %\forall A\subseteq X$.
	\item If $A\subseteq B\subseteq X$ then  $d_{X}(a,B)\leq d_{X}(a,A),\forall a\in X.$
	\item Let $A$ and $B$ be subsets of metric spaces $X$ and $Y$, respectively. Then $\partial_{X\times Y}(A\times B)= (\overline{A}\times \partial_{Y}(B)\cup(\partial_{X}(A)\times \overline{B})$. Similarly, for subsets $A_i$ of metric spaces $X_i$, we have $\partial_{\prod\limits_{1}^{n}X_i}(\prod\limits_{1}^{n}A_i)=\bigcup\limits_{1}^{n}(\overline{A}_1\times \overline{A}_2\times...\times\partial_{X_i}(A_i)\times...\times \overline{A}_k)$.
	\item For  separated subsets $A$ and $B$ of a metric space $X$, we have
%	if$\overline{A} \cap \overline{B} = \emptyset,$ then 
	$ \partial_{X}(A \cup B)= \partial_{X}(A) \cup \partial_{X}(B)$.
	\item Let $(X,d_X)$ be a path metric space \cite{Gro99}. For a proper subset $A$ of $X$,  $d_{X}(a,\partial_{X}(A))= d_{X}(a,A^c),\forall a\in A$.
\end{itemize}

\section{\textbf{The Definition of Center and Radius}}

 Let $(X,d_{X})$ be a metric space and $A$ be a subset of X. We introduce the notion of center and radius of A in X:
 
\begin{definition}[Center of a subset]
    The center of subset $A$ of a metric space $X$ is the set
$ \{a \in A | d_{X}(a,\partial_{X}(A)) \geq d_{X}(b,\partial_{X}(A)), \forall\, b \in A\}$, where $\partial_{X}(A)$ is the boundary of $A$ in $X$. We denote the center of $A$ in $X$ by $Cent_{X}(A).$

\end{definition}

 Thus the center of A is the set of all those elements of $A$ which are at the maximum distance from the boundary of $A$.

\begin{definition}[Radius of a subset]
    The radius of subset $A$ of a metric space $X$ is the distance between the center of $A$ and the boundary of $A$.  We denote the radius of A in X by $rad_{X}(A).$ 
    %So, we have $rad_{X}(A)= d_{X}(Cent_{X}(A),\partial_{X}(A)).$
\end{definition}

    It is clear that for every point in $Cent_{X}(A)$ has the same distance from $\partial_{X}(A)$. Thus $rad_{X}(A)= d_{X}(Cent_{X}(A),\partial_{X}(A)) = d_{X}(a, \partial_{X}(A)) , \forall a\in Cent_{X}(A)$.

\begin{example}\label{2.3}
{\normalfont     Let $\mathbb{R}$ be the set of real numbers with usual metric. For subsets $A=\left[0,1\right], B=\left[0,1\right] \cup \left[2,3\right]$ and $C=\left[0,1\right] \cup \left[5,10\right]$, we get $Cent_{\mathbb{R}}$(A)$ =\{0.5\}$ $\&$ $rad_{\mathbb{R}}(A) = 0.5$, $Cent_{\mathbb{R}}$(B)$ =\{0.5,2.5\}$ $\&$ $rad_{\mathbb{R}}(B) = 0.5$ and $Cent_{\mathbb{R}}$(C)$ =\{7.5\}$ $\&$ $rad_{\mathbb{R}}(C)= 2.5$.}
%     \begin{itemize}
%         \item[$(i)$] If $A=\left[0,1\right]$ then $Cent_{\mathbb{R}}$(A)$ =\{0.5\}$ and $rad_{\mathbb{R}}(A) = 0.5$.
 %        \item[$(ii)$] If $B=\left[0,1\right] \cup \left[2,3\right]$ then $Cent_{\mathbb{R}}$(B)$ =\{0.5,2.5\}$ and $rad_{\mathbb{R}}(B) = 0.5$.
  %       \item [$(iii)$] If $C=\left[0,1\right] \cup \left[5,10\right]$ then $Cent_{\mathbb{R}}$(C)$ =\{7.5\}$ and $rad_{\mathbb{R}}(C)= 2.5$.
 %    \end{itemize}    
\end{example}

\begin{example}\label{2.4}
     {\normalfont Let $\mathbb{R}^{2}$ be the real plane with the Euclidean metric.  For the unit disc $\mathbb{D}^{2} \subseteq \mathbb{R}^{2}$,  $Cent_{\mathbb{R}^{2}}(\mathbb{D}^{2}) =\{\left(0,0\right)\}$ $\&$ $rad_{\mathbb{R}^{2}}(\mathbb{D}^{2}) = 1,$ and 
    for a punctured unit disc $A= \mathbb{D}^{2}\backslash\{(0,0)\} \subseteq \mathbb{R}^{2}$, $Cent_{\mathbb{R}^{2}}(A) =\{\left(x,y\right)\in A|x^{2}+y^{2}=(\frac{1}{2})^{2}\}$ $\&$ $rad_{\mathbb{R}^{2}}(A) = \frac{1}{2}.$}
\end{example}

\begin{example}\label{3.5}
	{\normalfont %The standard n-simplex  $\Delta^n$ is the subset of $R^{n+1}$ given by $\Delta^n=\{(t_0,t_1,...,t_n)\in \mathbb{R}^{n+1}|\sum\limits_{i=0}^{n}t_i=1\, \text{and}\, t_i \geq 0,\, \text{for}\, 0\leq i\leq n \}$. 
		The $n + 1$ vertices of the standard $n$-simplex $\Delta^n$ are the points $e_i$, $1\leq i \leq n+1,$ in the Euclidean space $\mathbb{R}^{n+1}$ whose $i$-th coordinate is 1 and all other coordinates are 0. The simplex $\Delta^n$ lies in the affine hyperplane $H^{n}\subseteq \mathbb{R}^{n+1}$ spanned by its vertices $e_i$.
		%A point $a$ in $\Delta^n$ is $a=\sum\limits_{i=1}^{n+1}e_it_i$ where $\sum\limits_{i=1}^{n+1}t_i=1\, \text{and}\, 0\leq t_i\leq 1, \text{for}\, 1\leq i\leq n+1.$ The  distance of $a$ from the boundary of $\Delta^n$ is the distance of $a$ from the nearest $(n-1)$-face of $\Delta^n$. If $a$ is the barycebter of $\Delta^n$ then $a$ has equal distance from each $(n-1)$-face of $\Delta^n$. If $a$ is not the barycenter of $\Delta^n$ then at least one of its barycenteric coordinate $t_i< \frac{1}{n+1}$ and its distance from $(n-1)$-face opposite to $e_i$ is strictly less than the distance of the barycenter of $\Delta^n$ from that face. Which means the distance of $a$ from the boundary of $\Delta^n$ is strictly less than the distance of the  barycenter of $\Delta^n$ from the boundary of $\Delta^n$. 
		The center of $Cent_{H^n}(\Delta^n)$ is the barycenter  $\frac{1}{n+1}(1,1,...,1)$ of $\Delta^n$.}
		% Then $Cent_{H^n}(\Delta^n)=\delta^n.$ 
\end{example}

 As the center of A  consists of all those points of A which are at the maximum distance from its boundary, $rad_{X}(A)$ is the maximum distance of any point $a \in A$ from its boundary. It is clear that if $Cent_X(A)\neq \emptyset$ then $rad_{X}(A) =
\sup_{a \in A}d_{X}(a,\partial_{X}(A))
$. And, if  $Cent_X(A) = \emptyset$ then $rad_{X}(A)=\infty$ but $\sup_{a \in A}d_{X}(a,\partial_{X}(A))$ could be finite.\par 
 It leads us to introduce the concept of the Semi-radius of a subset $A$ in the metric space $X$.

\begin{definition}[Semi-radius of a subset]
	 The Semi-radius of subset $A$ of a metric space $X$ is the supremum of the set that consists of distance of any point $a\in A$ from the boundary of $A$. We denote the Semi-radius of A in X by $Srad_{X}(A).$\par That is, $Srad_{X}(A)= \sup_{a \in A}d_{X}(a,\partial_{X}(A))$.
\end{definition}

Note that $rad_{X}(A) \geq Srad_{X}(A)$. We can notice it from the following examples.
 
 \begin{example}\label{3.6.}
 	{\normalfont Let $X=\mathbb{R}\backslash\{1\}$ be the  metric subspace of Euclidean line $\mathbb{R}$. Let $A=[0,2]\backslash\{1\}$ be a subset of $X.$ Then $\partial_{X}(A)=\{0,2\}$. If $Cent_{X}(A)\neq \emptyset$, then 
 		%for every $x\in Cent_{X}(A)$ 
 		$\exists$ $a \in A$ such that $d_{X}(a,\partial_{X}(A)) \geq d_{X}(x,\partial_{X}(A)) ,\forall x\in A$, which is not true. So, $Cent_{\mathbb{R}}(A)=\emptyset$ and
 		$rad_{\mathbb{R}}(A)= \infty$, whereas  $Srad_{X}(A)=1<rad_{X}(A).$
 	}
 \end{example}
 
 \begin{example}\label{3.7.}
 	 {\normalfont Let $\mathbb{R}$ be the set of real numbers with usual metric. For $A=\bigcup\limits_{n\in\mathbb{N}} [n+\frac{1}{n},n+1]\subseteq\mathbb{R}$, we have $Cent_{\mathbb{R}}(A)=\emptyset$ and $rad_{\mathbb{R}}(A)=\infty$ but $Srad_{\mathbb{R}}(A)=\frac{1}{2}.$
 }
 \end{example}

\begin{example}\label{2.6}
      {\normalfont Let $X=(-\infty,0)\cup (\mathbb{Q}\cap\left[0,\pi\right])\cup[\pi,\infty)$ be a metric subspace of the Euclidean line $\mathbb{R}$. For $A= (\mathbb{Q}\cap\left[0,\pi\right])\subseteq X,$ we have $\partial_{X}(A)=\{0,\pi\}$, $Cent_{X}(A)=\emptyset$ and $rad_{X}(A)=\infty$ but $Srad_{X}(A)=\frac{\pi}{2} < rad_{X}(A).$ }
\end{example}

\begin{example}
	{ \normalfont Let $I_n=[0,1], n\in \mathbb{N}$ be intervals. Take a disjoint union $X=\sqcup_{n\in \mathbb{N}}I_{n}$ and define $d_{X}(a_i, b_j)=\frac{1}{2}(1-\delta_i^j)+ (2-\delta_i^j)|a-b|,$ where $a_i \in I_i, b_j\in I_j$, $\forall i,j\in \mathbb{N}$  
	%$(a,b \in\mathbb{R}$ and $i,j$ denotes the index of the interval)
	 and $\delta_i^j$ denote the kronecker delta. Then $(X,d_{X})$ is a metric space. For $A=\sqcup_{n=3}^\infty\left[\frac{1}{n},1-\frac{1}{n}\right]\subseteq X$, we have $\partial_{X}(A)= \sqcup_{n=3}^\infty\{\frac{1}{n},1-\frac{1}{n}\}$, $Cent_{X}(A)=\emptyset$, $rad_{X}(A)=\infty$ but $Srad_{X}(A)=\frac{1}{2}<rad_{X}(A).$}
\end{example}

Notice that if boundary of a subset $A$ of a metric space $X$ is empty then every point of $A$ is at infinite distance from $\partial_{X}(A)$. As the boundary of any metric space X is empty in itself, we have $Cent_{X}(X) =X$ $\&$ $rad_{X}(X)=\infty.$ Similarly, for the empty set $\emptyset$, $Cent_{X}(\emptyset) =\emptyset$ $\&$ $rad_{X}(\emptyset)=\infty.$ 
%For any metric space $(X,d_{X})$, $Cent_{{X}}$(X)$ =$X$ $, as the boundary of X in ($X$,$d_{X}$) is empty so the distance of every point of $X$ from its boundary is maximum, and $rad_{{X}}(X) = \infty$  as boundary of X is empty so the distance of $Cent_{X}(X)$ from empty set is infinite.
 %similarly, the center of the empty set $\emptyset$ is also empty, that is $Cent_{\textbf{X}}(\phi) =\phi$, and $rad_{{X}}(\phi) =\infty$.
 
 \begin{lemma}\label{2.8}
    For any clopen subset $A$ of a metric space $X$, $Cent_{X}(A)= A$ and $rad_{{X}}(A) = \infty$. 
 \end{lemma}
 
% \begin{example}
  %  If $(\textbf{X},d_{d})$ be a discrete metric space then for any subset A of $\mathbf{X} , Cent_{\textbf{X}}(A) =A$  and $rad_{{X}}(A) = \infty$.
%\end{example}

 But if a subset A of a metric space X has infinite radius then it does not mean that A is clopen in X. Consider a subset $A=\left[0,\infty\right)$ of the set $\mathbb{R}$ of real numbers with the usual metric. Then, $\partial_{\mathbb{R}}(A)=\{0\}$. 
 %If $Cent_{\mathbb{R}}(A)\neq \emptyset$ then $\exists$ $a \in A$ such that $a \geq x ,\forall x\in A$, which is not true. So,
 Here, $Cent_{\mathbb{R}}(A)=\emptyset$ and
  $rad_{\mathbb{R}}(A)= Srad_{\mathbb{R}}(A)=\infty$.\\

 The following result is for nonclopen subsets of a metric space.
 
\begin{theorem}\label{3.6}
     Let $A$ be a nonclopen subset of a metric space X. Then  $Cent_{X}(A)$ is nonempty if and only if $rad_{X}(A)$ is finite.
     \begin{proof}
     	As $A \subseteq X$ is nonclopen, $\partial_{X}(A) \neq \emptyset$. If  $Cent_{X}(A)$ is nonempty then $rad_{X}(A)=d_{X}(Cent_{X}(A),\partial_{X}(A))=\inf_{b\in \partial_{X}(A)}d_{X}(a,b)\leq d_{X}(a,b),\forall a\in Cent_{X}(A),\\ \forall b\in \partial_{X}(A)$, which is finite. \par
    % We know that in a metric space ($X$,$d_{X}$) a subset A is clopen iff $\partial_{X}(A) = \emptyset$. As here $A \subseteq X$ is not clopen $\implies$ $\partial_{X}(A) \neq \emptyset$. A is bounded which means $d_{X}(a,b) \leq k,  \forall a,b\in A$ for some positive number k, i.e. distance between any two points of A is bounded above by k. If A is bounded then $\overline{A}$ is also bounded. So distance between any two points of $\overline{A}$ is bounded above by some positive number $k^{'}$. And $\partial_{X}(A) \subseteq \overline{A}$. So the distance of any point of $A(\subseteq \overline{A})$ which is at the maximum distance from $\partial_{X}(A)$ is also less than that number $k^{'}$ $\implies$ $rad_{X}(A)$ is less than $k^{'}$ $\implies$ $rad_{X}(A)$ is finite.
    Conversely, if $rad_{X}(A)$ is finite then by the definition of radius, the center of $A$ is nonempty.
     \end{proof}
\end{theorem}

Next, we discuss examples of subsets of a metric space with zero radius.
%We give an example of an unbounded set having finite radius. 
Consider $\mathbb{N} \subseteq \mathbb{R}$, the set of real numbers with usual metric. Then $Cent_\mathbb{R}(\mathbb{N})=\mathbb{N}$ and $rad_\mathbb{R}(\mathbb{N})=0$. Infact, for any totally disconnected subset A of $\mathbb{R}$, we get $Cent_{\mathbb{R}}(A)= A$ and $rad_{\mathbb{R}}(A)=0.$ For the unit circle $\mathbb{S}^{1}$ in the Euclidean plane $\mathbb{R}^{2}$, $Cent_{\mathbb{R}^{2}}(\mathbb{S}^{1}) = \mathbb{S}^{1}$ $\&$ $rad_{\mathbb{R}^{2}}(\mathbb{S}^{1}) = 0.$   We know that topological manifolds are metrizable spaces. Let $N$ be an $n$-dimensional submanifold of a topological manifold $M$ of dimension $m$ where $n<m.$ Note that $N\subseteq \partial_{M}(N)$, and hence $Cent_{M}(N)=N$ and $rad_{M}(N)=0$. Also notice that in Example \ref{3.5}, $Cent_{\mathbb{R}^{n+1}}(\Delta^n)=\Delta^n$ and $rad_{\mathbb{R}^{n+1}}(\Delta^n)=0$.\par
In general, for any subset $A$ of a metric space $X$ contained in its boundary, it means for $A$ having empty interior, we have

 \begin{lemma}\label{A}
   Let $A$ be a nonempty subset of a metric space X such that $A$ has empty interior. Then, $Cent_{X}(A)= A$ and $rad_{X}(A)= 0$.
\end{lemma}

    The following result is for subsets  of metric spaces having nonempty interior.

\begin{lemma}\label{B}
    Let X be a metric space and $A \subseteq X$ such that $A$ has nonempty interior. Then, $Cent_{X}(A)$ $\subseteq$ $A^\circ$.
\end{lemma}

\begin{proof}
	If $A$ is clopen then it is true by Lemma \ref{2.8}. And, if $A$ is nonclopen then for any $a \in A$, either $a \in \partial_{X}(A)$ or $a \in A^\circ$. If $a \in \partial_{X}(A)$ then $d_{X}(a,\partial_{X}(A))=0$. If $a \in A^\circ$ then $\exists$ $\epsilon > 0 $ such that $B_{d_{X}}(a, \epsilon) \subseteq A^\circ$. As  $B_{d_{X}}(a,\epsilon) \cap \partial_{X}(A) =\emptyset$, we have $d_{X}(a,\partial_{X}(A)) \geq \epsilon >0, \forall a\in A^\circ.$ So, by definition of center of $A$ in $X$, we get $Cent_{X}(A) \subseteq A^\circ$.
	% As interior points of $A$ are positive distance away from $\partial_{X}(A)$ where as boundary points are zero distance away from $\partial_{X}(A)$.  This implies that  $Cent_{X}(A) \subseteq A^\circ$.
\end{proof}
     
\begin{theorem}\label{C}
    Let X be a metric space and $A$ be a nonempty subset of X. Then, $A^\circ = \emptyset$ if and only if $rad_{X}(A) = 0$.
\end{theorem}
\begin{proof}
      If $A^\circ = \emptyset$ then
      % $A \subseteq \partial_{X}(A)$. Then 
      by Lemma \ref{A}, $rad_{X}(A)= 0$. Conversely, assume that $rad_{X}(A)= 0$ then by Lemma \ref{2.8}, $A$ is nonclopen and by Theorem \ref{3.6}, $Cent_{X}(A)$ is nonempty. So, for $x \in Cent_{X}(A),$ we get $d_{X}(x, \partial_{X}(A)) = 0 \implies x \in \overline{\partial_{X}(A)}= \partial_{X}(A)$, and hence we get $Cent_{X}(A) \subseteq \partial_{X}(A)$. 
     If $A^\circ \neq \emptyset$ then by Lemma \ref{B} $Cent_{X}(A) \subseteq A^\circ.$ As $\partial_{X}(A) \cap A^\circ=\emptyset,$ we get $Cent_{X}(A)=\emptyset,$ a contradiction. Thus, $A^{\circ}=\emptyset.$
\end{proof}

% Also note that, for any subset A of a metric space X we have, $A^{\circ} \neq \phi \Longleftrightarrow rad_{X}(A) > 0$.\\
%For any totally disconnected subset A of $\mathbb{R}$, real numbers with usual metric. We have $A^\circ=\phi$, so by Theorem \ref{C}, $Cent_{\mathbb{R}}(A)= A$ and $rad_{\mathbb{R}}(A)=0.$\\
%For any subset A of a metric space X, we have $Cent_{X}(A) \subseteq A \subseteq X$. Note that $Cent_{X}(A)$ has the following properties as a subset of X:

%\begin{theorem}
   % Let $A$ be a subset of a metric space X such that $\partial_{X}(A)\neq\phi.$ Then, the following statements are equivalent: $rad_{X}(Cent_{X}(A)) = 0$ $\Longleftrightarrow$ $(Cent_{X}(A))^{\circ}$ $=\phi$ $\Longleftrightarrow$ $Cent_{x}(A) \subseteq \partial_{X}(Cent_{X}(A))$ $\Longleftrightarrow$ $Cent_{X}(Cent_{X}(A)) = Cent_{X}(A)$.

%\end{theorem}

%\begin{proof}
%        By Theorem \ref{C},  $rad_{X}(Cent_{X}(A))= 0 \Longleftrightarrow (Cent_{X}(A))^{\circ}= \phi$. And interior of a set is empty iff each of its point is a boundary point, so $(Cent_{X}(A))^{\circ}= \phi \Longleftrightarrow Cent_{x}(A) \subseteq \partial_{X}(Cent_{X}(A))$ then by Lemma \ref{A}, $Cent_{x}(A) \subseteq \partial_{X}(Cent_{X}(A)) \implies Cent_{X}(Cent_{X}(A)) = Cent_{X}(A), rad_{X}(Cent_{X}(A)) = 0$.
%\end{proof}

\begin{theorem}
    The center of a subset $A$ of a metric space $X$ is closed in A.
\end{theorem}
\begin{proof}
     Let $b \in A$ be a limit point of $Cent_{X}(A)$ in A. Then, $\exists$ a sequence $(a_{n})$ in $Cent_{X}(A)$ such that $(a_{n}) \longrightarrow b$. Consider a map $p: A \longrightarrow \mathbb{R}$ such that $p(x) = d_{X}(x, \partial_{X}(A)), \forall x \in A$. It is easy to observe that $p$ is a continuous map. By the continuity of $p,$ we have $p(a_{n})$ $\longrightarrow p(b)$. As $a_{n} \in Cent_{X}(A),$ which means $p(a_{n})= d_{X}(a_{n}, \partial_{X}(A))= rad_{X}(A), \forall n \in \mathbb{N}.$ So, $p(a_{n})$ is a constant sequence, and hence it converges to $rad_{X}(A)$. Thus, $p(b) = d_{X}(b,\partial_{X}(A)) = rad_{X}(A).$ This implies that $b \in Cent_{X}(A)$. Hence, $Cent_{X}(A)$ is a closed subset of A.
\end{proof}

It is not necessary that  $Cent_{X}(A)$ is a closed subset of X. For example: Consider a subset 
$A= (0,1)\times \{0\}$ of the Euclidean plane $\mathbb{R}^{2}$. Then $Cent_{\mathbb{R}^{2}}(A)= A$ which is not closed in $\mathbb{R}^{2}.$\\

Now for any subset A of a metric space X, we establish a relationship between the radii of  $\overline{A}$ and $A^\circ$ with the radius of $A$.
% If $A^\circ=\emptyset$ then by Lemma \ref{A}, $rad_{X}(A)=0$, which means it is always less than or equals to $rad_{X}(A^\circ)$ and $rad_{X}(\overline{A})$.

%\begin{theorem}
% Let $X$ be a metric space and $A \subseteq X$ such that $A^\circ \neq \phi$ then $Cent_{X}(A) \subseteq Cent_{X}(\overline{A})$. Moreover $A, \overline{A}$ $\&$ $A^{\circ}$ are cocentric and $rad_{X}(A^\circ) = rad_{X}(A) = rad_{X}(\overline{A})$.
%\end{theorem}
\begin{theorem}\label{3.28}
	Let $X$ be a metric space and $A \subseteq X$ such that  $Cent_{X}(A)\neq \emptyset$. Then $rad_{X}(A)\leq rad_{X}(A^\circ)$ and $rad_{X}(A)\leq rad_{X}(\overline{A}).$ %Moreover, $Cent_{X}(\overline{A})\subseteq Cent_{X}(A)$.
\end{theorem}

\begin{proof} %If $A$ is clopen then $A= A^\circ= \overline{A}$ then the result is trivialy true. If $A$ is non clopen then by Lemma \ref{C}, $rad_{X}(A)$ is finite. %$Cent_{X}(A)\subseteq A^\circ.$

If $Cent_{X}(A^\circ)$ is empty then $rad_{X}(A^\circ)=\infty$, which implies $rad_{X}(A)\leq rad_{X}(A^\circ)$. If $Cent_{X}(A^\circ)\neq\emptyset$ then for
 $b\in Cent_{X}(A^\circ)$, we have $ d_{X}(b,\partial_{X}(A^\circ)) \geq d_{X}(b',\partial_{X}(A^\circ)),\forall b'\in A^\circ.$ As $\partial_{X}(A^\circ)\subseteq\partial_{X}(A),$ we get $d_{X}(b',\partial_{X}(A^\circ))\geq d_{X}(b',\partial_{X}(A))$, $\forall b'\in A^\circ\subseteq A.$ Hence, $ d_{X}(b,\partial_{X}(A^\circ))\geq d_{X}(b',\partial_{X}(A)),\forall b'\in A^\circ.$ This implies that $rad_{X}(A^\circ)\geq \sup_{b'\in A^\circ}d_{X}(b',\partial_{X}(A))$ and for any $b'\in \partial_{X}(A),$ we have $ d_{X}(b',\partial_{X}(A))=0$. So, $ rad_{X}(A^\circ)
 %= \sup_{b'\in A^\circ}d_{X}(b',\partial_{X}(A))
 \geq \sup_{b'\in A}d_{X}(b',\partial_{X}(A))
=rad_{X}(A).$ Thus, $rad_{X}(A)\leq rad_{X}(A^\circ)$.\par 
Similarly, as $\partial_{X}(\overline{A})\subseteq\partial_{X}(A)$, we get $rad_{X}(A)\leq rad_{X}(\overline{A}).$ 
%The equality holds when $Cent_{X}(\overline{A})\neq \emptyset$.
%Next, if $Cent_{X}(\overline{A})$ is empty then $rad_{X}(\overline{A})=\infty$, which implies $rad_{X}(A)\leq rad_{X}(\overline{A})$. If $Cent_{X}(\overline{A})\neq\emptyset$ then for	$b\in Cent_{X}(\overline{A})$ we have $ d_{X}(b,\partial_{X}(\overline{A})) \geq d_{X}(b',\partial_{X}(\overline{A})),\forall b'\in \overline{A}.$ Because $\partial_{X}(\overline{A})\subseteq\partial_{X}(A)$ so we have $d_{X}(b',\partial_{X}(\overline{A}))\geq d_{X}(b',\partial_{X}(A)),\forall b'\in A.$ Hence, $ d_{X}(b,\partial_{X}(\overline{A}))\geq d_{X}(b',\partial_{X}(A)),\forall b'\in A.$ This implies that $rad_{X}(\overline{A})\geq \sup_{b'\in A}d_{X}(b',\partial_{X}(A))	=rad_{X}(A).$ 
%	Next, let $b\in Cent_{X}(\overline{A})$. Then $rad_{X}(\overline{A})\geq d_{X}(b',\partial_{X}(\overline{A}))\geq d_{X}(b',\partial_{X}(A)),$ $\forall b'\in A.$ Thus  $rad_{X}(\overline{A})\geq \sup_{b'\in A}(d_{X}(b',\partial_{X}(A)))=rad_{X}(A).$
	%    Now, let $a \notin cent_{X}(A) \implies d_{X}(a,\partial_{X}(A))<rad_{X}(A).$ So, $B(a,rad_X(A))\cap \partial_{X}(A)\neq\phi$ $\implies B(a,rad_X(A))\cap \partial_{X}(\overline{A})\neq\phi$, because $\partial_{X}(\overline{A})\subseteq\partial_{X}(A).$ As $rad_{X}(\overline{A})\geq rad_{X}(A) \implies B(a,rad_X(\overline{A}))\cap \partial_{X}(\overline{A})\neq\phi \implies a\notin Cent_{X}(\overline{A}).$ So we have $Cent_{X}(\overline{A})\subseteq Cent_{X}(A)$.
\end{proof}

%	If $Cent_{X}(A)=\emptyset$  then $Cent_{X}(A^\circ)$ and $Cent_{X}(\overline{A})$ are also empty. Because if $Cent_{X}(A^\circ)\neq \emptyset$, then for $b\in Cent_{X}(A^\circ)$, we have $ d_{X}(b,\partial_{X}(A^\circ)) \geq d_{X}(b',\partial_{X}(A^\circ)),\forall b'\in A^\circ.$ Here $b\in Cent_{X}(A^\circ)\subseteq A^\circ\subseteq A$ and as $Cent_{X}(A)=\emptyset$ so $\exists c\in A$ such that $d_{X}(c,\partial_{X}(A))> d_{X}(b,\partial_{X}(A))$ and $c\notin \partial_{X}(A) \implies c\in A^\circ $. As $\partial_{X}(A^\circ)\subseteq\partial_{X}(A)$ so, $d_{X}(c,\partial_{X}(A^\circ))\geq d_{X}(c,\partial_{X}(A))> d_{X}(b,\partial_{X}(A))$.\par
 %   If $Cent_{X}(\overline{A}))\neq \emptyset$, then for $b\in Cent_{X}(\overline{A})$, we have $ d_{X}(b,\partial_{X}(\overline{A})) \geq d_{X}(b',\partial_{X}(\overline{A})),\forall b'\in \overline{A}.$ Here $b\in \overline{A}=A\cup \partial_{X}(A)$ and as $Cent_{X}(A)=\emptyset$ and $\partial_{X}(\overline{A})\subseteq\partial_{X}(A)$.So, if $b\in A$ then $\exists c\in A$ such that $d_{X}(c,\partial_{X}(\overline{A}))\geq d_{X}(c,\partial_{X}(A))>d_{X}(b,\partial_{X}(A))\geq d_{X}(b,\partial_{X}(\overline{A}))$ . A contradiction. And, if $b\in \partial_{X}(A)$

\begin{example}
	{\normalfont Consider the set $\mathbb{R}$ of real numbers with the usual metric.\\
		 Let $A=\{\frac{1}{n} |n \in \mathbb{N}\}$ $\subseteq \mathbb{R}.$ Then $\overline{A}=\{0\} \cup \{\frac{1}{n} |n \in \mathbb{N}\}$ and $A^\circ = \emptyset$. We get $Cent_{\mathbb{R}}(A^\circ)= \emptyset$ $\&$ $Cent_{\mathbb{R}}(A)= A$,   $Cent_{\mathbb{R}}(\overline{A})= \overline{A}$ 
	and $rad_{\mathbb{R}}(A^\circ) = \infty$ $\&$ $rad_{\mathbb{R}}(A) =rad_{\mathbb{R}}(\overline{A}) = 0$.\\
	%Notice that $rad_{X}(A^\circ)\neq rad_{X}(A),$ and $Cent_{X}(\overline{A}) \not\subseteq Cent_{X}(A)$ and $ Cent_{X}(A) \neq Cent_{X}(A^\circ)$.
	Let $B=(\mathbb{Q}\cap [0,\infty))\subseteq\mathbb{R}$. Then $B^\circ=\emptyset$ and $\overline{B}=[0,\infty)$. We get $Cent_{\mathbb{R}}(B^\circ)= \emptyset$ $\&$ $Cent_{\mathbb{R}}(B)= B$,   $Cent_{\mathbb{R}}(\overline{B})= \emptyset$ 
	and $rad_{\mathbb{R}}(B) = 0$ $\&$ $rad_{\mathbb{R}}(B^\circ) =rad_{\mathbb{R}}(\overline{B}) = \infty$.\\ 
	Let $C=[0,1)\subseteq\mathbb{R}$ then $rad_{\mathbb{R}}(C)= rad_{\mathbb{R}}(C^\circ) =rad_{\mathbb{R}}(\overline{C}) =\frac{1}{2}$.}
\end{example}

In the Euclidean plane  $\mathbb{R}^{2}$, let $Y=\mathbb{R}\times \{0\}$  and $A=(0,1)\times \{0\}.$ Here $A\subseteq Y\subseteq \mathbb{R}^2$ and $Cent_{X}(A)\neq\emptyset$. 
		%$Cent_{\mathbb{R}^2}(A)= A, Cent_{Y}(A)=\{(0.5,0)\},$
 Notice that 		$rad_{\mathbb{R}^2}(A)= 0\leq rad_{Y}(A)= \frac{1}{2}$. In general, we have the following result.

	\begin{theorem}
	Let $X$ be a metric space and $Y$ be a subspace of X. If $A \subseteq Y \subseteq X$ such that $Cent_{X}(A)$ is nonempty,  then $rad_{X}(A) \leq rad_{Y}(A).$
\end{theorem}
\begin{proof} 
	%If $A$ is a clopen subset of $X$ then it is also clopen in $Y$ then by Theorem \ref{2.8} both of these radius will be infinite. If $A$ is non clopen then by Theorem \ref{3.6} $rad_{X}(A)$ is finite.
		We know that $\partial_{Y}(A)\subseteq \partial_{X}(A)\subseteq X.$ So, $ d_{X}(a,\partial_{X}(A))\leq d_{X}(a,\partial_{Y}(A)) ,\forall\, a \in A \implies d_{X}(a,\partial_{X}(A))\leq d_{Y}(a,\partial_{Y}(A)),\forall\, a\in A.$ And, if $Cent_{X}(A)\neq\emptyset$ , we have $rad_{X}(A)=\sup_{a\in A}d_{X}(a,\partial_{X}(A))\leq \sup_{a\in  A}d_{Y}(a,\partial_{Y}(A))\leq rad_{Y}(A).$
	%	If $Cent_{X}(A)=\emptyset$ then $Cent_{Y}(A)=\emptyset$. As if $b\in Cent_{Y}(A)\subseteqA$ , as $Cent_{X}(A)=\emptyset$ so $\exists c\in A$ such that $d_{X}(c,\partial_{X}(A))>d_{X}(b,\partial_{X}(A)).$ So, we get $d_{Y}(c,\partial_{Y}(A))=d_{X}(c,\partial_{Y}(A))\geq d_{X}(c,\partial_{X}(A))>d_{X}(b,\partial_{X}(A))$, a contradiction. So $Cent_{Y}(A)=\emptyset$. Hence the result.
	%Now, we need to show $Cent_{Y}(A) \subseteq Cent_{X}(A)$(Because if $A\subseteq B \subseteq X$ then $d_{X}(x,A)\geq d_{x}(x,B),\forall x\in X$.)
	%Let $a\in Cent_{Y}(A) \implies d_{Y}(a,\partial_{Y}(A))\geq  d_{Y}(b,\partial_{Y}(A)), \forall b\in A.$ (As metric $d_{Y}$ on Y is induced from metric $d_{X}$ of X. $\implies d_{Y}(a,b)=d_{X}(a,b),\forall a,b \in Y$ and $d_{Y}(a,A)=d_{X}(a,A), \forall a \in Y, \forall A \subseteq Y.$) $\implies rad_{Y}(A)=d_{Y}(a,\partial_{Y}(A))\geq d_{Y}(b,\partial_{Y}(A))=d_{X}(b,\partial_{Y}(A)),\forall b \in A.$By equation(21), $\implies rad_{Y}(A) \geq d_{X}(b,\partial_{X}(A)),\forall b \in A. \implies rad_{Y}(A) \geq \sup_{b \in A}(d_{X}(b,\partial_{X}(A)))=rad_{X}(A). \implies rad_{Y}(A) \geq rad_{X}(A).$\\
	%Now, we show that $Cent_{Y}(A) \subseteq Cent_{X}(A) \cap Y.$ As  $Cent_{Y}(A) \subseteq A\subseteq Y$, we just need to show $Cent_{Y}(A) \subseteq Cent_{X}(A)$. 
\end{proof}

\begin{remark}
{\normalfont In the above result, $Cent_{X}(A)\neq\emptyset$ is necessary. In Example \ref{2.6}, let $Y= (\mathbb{Q}\cap\left[0,\pi\right])\cup[\pi,\infty).$ Then $A\subseteq Y\subseteq X$. Here $\partial_{Y}(A)=\{\pi\}$, $Cent_{Y}(A)=\{0\}$ and $Cent_{X}(A)=\emptyset$. But $rad_{Y}(A)=\pi$ whereas $rad_{X}(A)=\infty$.}
\end{remark}
%\begin{theorem}[Characteristics of Center]
  % $$ Let $A \subseteq X$ be a  non clopen subset of a metric space $(X, d_{X})$. Then $Cent_{X}(A)$ has the following properties,
%\end{theorem}
Let $f:X\longrightarrow Y$ be an isometry between two metric spaces $X$ and $Y$ \cite{Robert}. We know that isometry preserves the boundary of a subset. Here, we observe that isometry also preserves the center and radius of a subset.\par

First, we prove the following lemma.
	\begin{lemma}\label{2.17}
   Let $f:X\longrightarrow Y$ be an isometry between two metric spaces $X$ and $Y.$ For a subset $A$ of $X$, $Cent_{X}(A)\neq\emptyset$ if and only if $Cent_{Y}(f(A))\neq\emptyset$.
	\end{lemma}
	\begin{proof}	As $f$ is an isometry, $ f(\partial_{X}(A))= \partial_{Y}(f(A)).$ For $a\in A$, we have $d_{X}(a,\partial_{X}(A))$\\ 
	$= d_{Y}(f(a),f(\partial_{X}(A))) = d_{Y}(f(a),\partial_{Y}(f(A))).$ Notice that $Cent_{X}(A)\neq \emptyset \iff \exists a\in A$ such that $d_{X}(a,\partial_{X}(A))\geq d_{X}(b,\partial_{X}(A)),\forall b\in A \iff d_{Y}(f(a),\partial_{Y}(f(A)))
	\geq d_{Y}(f(b),\partial_{Y}(f(A)),\forall b\in A. \iff f(a)\in Cent_{Y}(f(A)) \iff Cent_{Y}(f(A))\neq\emptyset. $ %Hence $Cent_{X}(A)=\emptyset$ if and only if $Cent_{Y}(f(A))=\emptyset$
		%Notice that $d_{X}(a,\partial_{X}(A))= d_{Y}(f(a),f(\partial_{X}(A)))=d_{Y}(f(a),\partial_{Y}(f(A)))\geq d_{Y}(f(b),\partial_{Y}(f(A))=d_{Y}(f(b),f(\partial_{X}(A))=d_{X}(a,\partial_{X}(A)),\forall b\in A.$ Which means $a\in Cent_{X}(A)$ if and only if $f(a)\in Cent_{Y}(f(A))$. Hence $Cent_{X}(A)=\emptyset$ if and only if $Cent_{Y}(f(A))=\emptyset$.
	\end{proof}
	
	\begin{theorem}\label{2.18}
		Let $f:X\longrightarrow Y$ be an isometry between two metric spaces $X$ and $Y.$ For a subset $A$ of $X$, $rad_{X}(A)= rad_{Y}(f(A))$ and $f(Cent_{X}(A))=Cent_{Y}(f(A)).$  
	\end{theorem}
	\begin{proof}
		%	As $f$ is an isometry, $ f(\partial_{X}(A))= \partial_{Y}(f(A)), \forall A \subseteq X.$
		 First, let $Cent_{X}(A)=\emptyset$. Then it is true by Lemma \ref{2.17}. Next, let $Cent_{X}(A)\neq \emptyset.$ Then $rad_{Y}(f(A))\geq d_{Y}(f(a),\partial_{Y}(f(A))=
		 d_{Y}(f(a),f(\partial_{X}(A)))
		 =d_{X}(a,\partial_{X}(A)),\, \forall a\in A.$ This implies that $rad_{Y}(f(A))\geq rad_{X}(A).$ By Lemma \ref{2.17}, $Cent_{Y}(f(A))\neq\emptyset.$ Similarly, we get $rad_{X}(A)\geq rad_{Y}(f(A)).$ Hence, $rad_{X}(A)= rad_{Y}(f(A)).$\\
		 As $rad_{Y}(f(A))=rad_{X}(A)=d_{X}(Cent_{X}(A),\partial_{X}(A))=
		d_{Y}(f(Cent_{X}(A)),\partial_{Y}(f(A))),$ we get every point of $f(Cent_{X}(A))$ is at the maximum distance from $\partial_{Y}(f(A))).$ So, $ 
		f(Cent_{X}(A))\subseteq Cent_{Y}(f(A)).$ Similarly, we get $Cent_{Y}(f(A))\subseteq f(Cent_{X}(A)).$ Hence, $f(Cent_{X}(A))= Cent_{Y}(f(A)).$
	\end{proof}
	\begin{remark}
		{\normalfont Let $X$ and $Y$ be two metric spaces such that $A\subseteq X$ and $B\subseteq Y$. If $Cent_{X}(A)$ is connected and $Cent_{Y}(B)$ is disconnected then by Theorem \ref{2.18}, there does not exist any isometry between $X$ and $Y$ such that $f(A)=B$.
			%induced from an isometry between $X$ and $Y$. %$f:X\longrightarrow Y$ such that $f(A)=B$.
	    % then by Theorem \ref{2.18} we have $f(Cent_{X}(A))= Cent_{Y}(B)$, a contradiction, as connectedness is preserved under continuity. So, there does not exist any isometry between $X$ and $Y$ such that $f(A)=B$. For example, Let $A$ be the unit disc centred at origin,  $\mathbb{D}^2$ in $\mathbb{R}^2$ and $B$ be the union of  $\mathbb{D}^2$ with the closed ball of radius 1 centred at $(1.5,0)$. Here, $Cent_{\mathbb{R}^2}(A)=\{(0,0)\}$ and $Cent_{\mathbb{R}^2}(B)=\{(0,0),(1.5,0)\}.$ So, there does not exist any isometry $f:\mathbb{R}^2 \longrightarrow \mathbb{R}^2$ such that $f(A)=B$.
	    }
	\end{remark}

	\section{\textbf{Center and Radius of a finite product of subsets of metric spaces}}
	%If $(X,d_{1})$ and $(X,d_{2})$ are equivalent metric spaces and $A \subseteq X$ then interior, closure and boundary of A are same in both space. similarly center of a subset is also same in equivalent metric spaces.
	%\begin{theorem}
	%	Let X be a nonempty set and $d_{1},d_{2}$ are two equivalent metrices on X. And $A \subseteq X$ be a subset of X, then $ Cent_{(X,d_{1})}(A)= Cent_{(X,d_{2})}(A).$
		 %$,rad_{(X,d_{1})}(A)= rad_{(X,d_{2})}(A).$
	%\end{theorem}
	
	Let $(X\times Y, d)$ be the product of metric spaces $(X,d_{X}).$ and $(Y,d_{Y}),$ where $d((x_{1},y_{1}),(x_{2},y_{2}))= \max\{d_{X}(x_{1},x_{2}),d_{Y}(y_{1},y_{2})\},\forall\,(x_1,y_1),(x_2,y_2)\in X\times Y.$ Now, we see how  $Cent_{X\times Y}(A\times B)$ is related to $Cent_{X}(A)$ and $Cent_{Y}(B)$.\\
	
	Let $A\subseteq X$ and $B\subseteq Y$ be subsets. If $rad_{X}(A)$ and $rad_{Y}(B)$ are infinite then there are three possible cases:
	% when both $rad_{X}(A)$ and $rad_{Y}(B)$ are infinite:
	(i) both $A$ and $B$ are clopen,  
	(ii) one of $A$ and $B$ is clopen and other has empty center, and
	%that is, if $A$ is clopen  and $Cent_{Y}(B)=\emptyset.$
	(iii) $Cent_{X}(A)= Cent_{Y}(B)=\emptyset.$

	\begin{theorem} \label{3.45}
	Let $(X,d_{X})$ and $(Y,d_{Y})$ be two metric spaces. Let $A\subseteq X$ and $B\subseteq Y$ be subsets.
   \begin{itemize}
	[noitemsep,nolistsep]
	\item[(i)] If both $A$ and $B$  are clopen, then $Cent_{X\times Y}(A\times B)=A\times B$ and $rad_{X\times Y}(A\times B)=\infty,$
	\item[(ii)] If $A$ is clopen and $Cent_{Y}(B)=\emptyset$,
			% is non clopen such that $rad_{Y}(B)$ is infinite,
    then $Cent_{X\times Y}(A\times B)=\emptyset$ and $rad_{X\times Y}(A\times B)=\infty,$ and
	\item[(iii)] %If $A \subseteq X$, $B \subseteq Y$ are such that
	 $Cent_{X}(A)= Cent_{Y}(B)=\emptyset$,
			%non clopen subsets such that their radii are infinite,
   then $Cent_{X \times Y}(A\times B)= \emptyset$ and $rad_{X \times Y}(A\times B)=\infty.$
		\end{itemize}
	\end{theorem}
	\begin{proof}
		($i$) If $A$ and $B$  are clopen then $A\times B$ is clopen, and hence $Cent_{X\times Y}(A\times B)=A\times B$ and $rad_{X\times Y}(A\times B)=\infty.$\\
		($ii$) As $A$ is clopen, $\partial_{X}(A)=\emptyset.$ 
		%and $B$ is non clopen such that $rad_{Y}(B)=\infty,$ we get $Cent_{Y}(B)=\emptyset.$ 
		Recall that $\partial_{X\times Y}(A\times B)= (\overline{A}\times \partial_{Y}(B))\cup(\partial_{X}(A)\times \overline{B})$. In this case, $\partial_{X\times Y}(A\times B)= A\times \partial_{Y}(B).$ First, we observe that $d((a,b),A\times \partial_{Y}(B))= d_{Y}(b,\partial_{Y}(B)),\forall (a,b)\in A\times B.$\\ 
		%Consider, for $(a,b)\in A\times B$ we get
		We have $d((a,b),A\times \partial_{Y}(B))\leq d((a,b),(a,y))=\max\{d_{X}(a,a),d_{Y}(b,y)\}= d_{Y}(b,y),\\
		\forall y\in \partial_{Y}(B)\implies d((a,b),A\times \partial_{Y}(B))\leq d_{Y}(b,\partial_{Y}(B)).$ And, for $x\in A$ and $y\in \partial_{Y}(B),$ we have  $d_{Y}(b,\partial_{Y}(B))\leq d_{Y}(b,y)\leq \max\{d_{X}(a,x),d_{Y}(b,y)\}=d((a,b),(x,y))\\
		\implies d_{Y}(b,\partial_{Y}(B))\leq d((a,b),A\times \partial_{Y}(B)).$ Hence, $d((a,b),A\times \partial_{Y}(B))= d_{Y}(b,\partial_{Y}(B)),\\ \forall (a,b)\in A\times B.$ \par
		%We get for $(a,b)\in A\times B,\, d((a,b),A\times \partial_{Y}(B))= \inf_{x\in A, y\in \partial_{Y}(B)}(\max\{d_{X}(a,x),d_{Y}(b,y)\})\leq \max\{d_{X}(a,a),d_{Y}(b,y)\}=d_{Y}(b,y), \forall y\in \partial_{Y}(B) \implies d((a,b),A\times \partial_{Y}(B))\leq \inf_{y\in \partial_{Y}(B)}d_{Y}(b,y)=d_{Y}(b,\partial_{Y}(B)).$ And, for $x\in A$ and $y\in \partial_{Y}(B),$ we have  $d_{Y}(b,y)\leq \max\{d_{X}(a,x),d_{Y}(b,y)\} \implies d_{Y}(b,y)\leq \inf_{x\in A}(\max\{d_{X}(a,x),d_{Y}(b,y)\})$ $\implies d_{Y}(b,\partial_{Y}(B))\leq \inf_{x\in A,y\in \partial_{Y}(B)}(\max\{d_{X}(a,x),d_{Y}(b,y)\})=d((a,b),A\times \partial_{Y}(B)).$ So we get  $d((a,b),A\times \partial_{Y}(B))=  d_{Y}(b,\partial_{Y}(B)),\forall(a,b)\in A\times B.$\\
		Suppose that $Cent_{X\times Y}(A\times B)\neq \emptyset.$  
		 For $(a,b)\in Cent_{X\times Y}(A\times B),$ we get $d_{Y}(b,\partial_{Y}(B))
		 =d((a,b),A\times \partial_{Y}(B))\geq d((a,b'), A\times \partial_{Y}(B))=d_{Y}(b',\partial_{Y}(B)),\forall b'\in B \implies b\in Cent_{Y}(B)$, a contradiction. So, $Cent_{X\times Y}(A\times B)=\emptyset$, and hence $rad_{X\times Y}(A\times B)=\infty.$\\
		($iii$) %As both $A$ and $B$ are non clopen and both have infinite radii, $Cent_{X}(A)=Cent_{Y}(B)=\emptyset.$ 
        %We have $d((a,b),\partial_{X\times Y}(A \times B))= \min\{d((a,b),(\partial_{X}(A) \times\overline{B})),d((a,b),(\overline{A}\times \partial_{Y}(B)))\},\forall (a,b)\in A\times B.$
         Similarly, as in case ($ii$), we get
	$ d((a,b),\overline{A}\times \partial_{Y}(B))=d_{Y}(b,\partial_{Y}(B)),$ and $d((a,b),\partial_{X}(A)\times \overline{B})=d_{X}(a,\partial_{X}(A)),\forall (a,b)\in A\times B.$\par
	Suppose that $Cent_{X\times Y}(A\times B)\neq \emptyset.$ , For $(a,b)\in Cent_{X\times Y}(A\times B)\subseteq A\times B,$ we get  $d((a,b),\partial_{X\times Y}(A\times B))= \min\{d((a,b),(\partial_{X}(A) \times\overline{B})),d((a,b),(\overline{A}\times \partial_{Y}(B)))\} =\min\{d_{X}(a,\partial_{X}(A)),d_{Y}(b,\partial_{Y}(B))\}.$ If $d((a,b),\partial_{X\times Y}(A\times B))=d_{X}(a,\partial_{X}(A))$ then $a \in Cent_{X}(A),$ a contradiction, and if $d((a,b),\partial_{X\times Y}(A\times B))=d_{Y}(b,\partial_{Y}(B))$ then $b \in Cent_{Y}(B),$ again a contradiction. So, $Cent_{X\times Y}(A\times B)=\emptyset$, and hence $rad_{X\times Y}(A\times B)=\infty.$
	\end{proof}
	
		\begin{example}
		\begin{itemize}{\normalfont
				\item[(i)] Let $A=\{2,3\}\subseteq\mathbb{Z}$ and $B=[0,\infty)\subseteq\mathbb{R},$ where $\mathbb{Z}$ is discrete space and $\mathbb{R}$ is equipped with the usual metric. As $A$ is clopen in $\mathbb{Z}$ and  $Cent_{\mathbb{R}}(B)=\emptyset,$ by Theorem \ref{3.45}($ii$), we get  $Cent_{\mathbb{Z}\times\mathbb{R}}(A\times B)=\emptyset$ and $rad_{\mathbb{Z}\times\mathbb{R}}(A\times B)=\infty.$
				\item[(ii)] Let $Q=(0,\infty)\times (0,\infty)$ be the first quadrant in $\mathbb{R}^2$ with the maximum metric. Then by Theorem \ref{3.45}($iii$), we get $Cent_{\mathbb{R}^2}(Q)=\emptyset$ and $rad_{\mathbb{R}^2}(Q)=\infty.$}
		\end{itemize}
	\end{example}
	
	\begin{theorem}\label{3.43}
		Let $(X,d_{X})$ and $(Y,d_{Y})$ be two metric spaces. For $A \subseteq X$ and $B \subseteq Y$, let $\Hat{B}=\{b \in B| d_{Y}(b,\partial_{Y}(B)) \geq rad_{X}(A)\}$. If $rad_{X}(A)\leq rad_{Y}(B)$,  then $Cent_{X \times Y}(A\times B)= Cent_{X}(A) \times \Hat{B}$. %\begin{itemize}
		%   \item[$(i)$] if $rad_{X}(A)< rad_{Y}(B)$, then $Cent_{X \times Y}(A\times B)= Cent_{X}(A) \times \Hat{B}$, and
		%   \item[$(ii)$] if $rad_{X}(A) = rad_{Y}(B)$, then $Cent_{X \times Y}(A\times B)= Cent_{X}(A) \times Cent_{Y}(B)$.
		%\end{itemize}  
	\end{theorem}
	\begin{proof} First, let both $rad_{X}(A)$ and $rad_{Y}(B)$ be infinite. In this case, we have three possibilities. If $A$ and $B$ are clopen then $Cent_{X}(A)=A$ and $\Hat{B}=B$, and the result follows by Theorem \ref{3.45}($i$).
	%If both $rad_{X}(A)$ and $rad_{Y}(B)$ are infinite such that atleast one of them is non clopen then
	And in other two cases either $\Hat{B}=\emptyset$ or $Cent_{X}(A)=\emptyset$, and the result  follows by Theorem \ref{3.45}($ii$) or Theorem \ref{3.45}($iii$).\par
   %Now, Assume that atleast one of $rad_{X}(A)$ and $rad_{Y}(B)$ is finite.
   Now, WLOG suppose that $rad_{X}(A)$ is finite, then $Cent_{X}(A)$ is nonempty. If $rad_{Y}(B)$ is finite, then $Cent_{Y}(B)\neq\emptyset$ and $Cent_{Y}(B)\subseteq \Hat{B} \implies \Hat{B}\neq\emptyset$. If $rad_{Y}(B)=\infty$ such that $rad_{X}(A)<Srad_{Y}(B)$, then  $\Hat{B}\neq\emptyset$.\par 
	%As $rad_{X}(A)\leq rad_{X}(B),$ we get $\Hat{B}\neq\phi.$Again, recall that for $(a,b)\in A\times B,$ we have $d((a,b),\partial_{X\times Y}(A \times B))= \min\{d((a,b),(\partial_{X}(A) \times\overline{B})),d((a,b),(\overline{A}\times \partial_{Y}(B)))\}.$
   In both the above cases, we observe that if $(a,b)\in Cent_{X}(A)\times \Hat{B},$ then $d((a,b),\\ \partial_{X\times Y}
   (A \times B))=rad_{X}(A).$
    As $a\in Cent_{X}(A)$ and $b\in \Hat{B}$,  $d_{X}(a,\partial_{X}(A))=rad_{X}(A)$ and $d_{Y}(b,\partial_{Y}(B))\geq rad_{X}(A)\implies d((a,b),\partial_{X\times Y}(A\times B))=\min\{d_{X}(a,\partial_{X}(A)),\\ d_{Y}(b,\partial_{Y}(B))\}=rad_{X}(A) $.
    Next, we observe that if $(a,b)\notin (Cent_{X}(A)\times \Hat{B}),$ then $d((a,b),\partial_{X\times Y}(A \times B))<rad_{X}(A).$ If $(a,b)\notin(Cent_{X}(A)\times \Hat{B}),$ then we have either $a\notin Cent_{X}(A)$ or $b\notin \Hat{B}$. If $a\notin Cent_{X}(A)$, then $d_{X}(a,\partial_{X}(A))<rad_{X}(A)$. And, if $b\notin \Hat{B}$, then $d_{X}(b,\partial_{X}(A))<rad_{X}(A).$ This implies $d((a,b),\partial_{X\times Y}(A \times B))<rad_{X}(A).$ So, $Cent_{X \times Y}(A\times B)= Cent_{X}(A) \times \Hat{B}$.\par 
    Now, if $rad_{Y}(B)=\infty$ such that $rad_{X}(A)\geq Srad_{Y}(B)$, then $\Hat{B}=\emptyset 
    %\implies Cent_{X}(A)\times \Hat{B}=\emptyset
    $. In this case, we prove that $Cent_{X\times Y}(A\times B)=\emptyset$. 
    Let $(a,b)\in Cent_{X\times Y}(A\times B)$. Then $ d((a,b),\partial_{X\times Y}(A \times B))=\min\{d_{X}(a,\partial_{X}(A)), d_{Y}(b,\partial_{Y}(B)\}$. If $d((a,b),\partial_{X\times Y}(A \times B))=d_{Y}(b,\partial_{Y}(B)$, then $b\in Cent_{Y}(B)$, a contradiction. So, $d((a,b),\partial_{X\times Y}(A \times B))=d_{X}(a,\partial_{X}(A)) \implies a\in Cent_{X}(A)$. So, we get $rad_{X}(A)=d_{X}(a,\partial_{X}(A))=d((a,b),\partial_{X\times Y}(A \times B))< d_{Y}(b,\partial_{Y}(B))\leq \sup_{b\in B}d_{Y}(b,\partial_{Y}(B))=Srad_{Y}(B)$, which is not the case. So, $Cent_{X\times Y}(A\times B)=\emptyset$.  Hence, our claim.
	\end{proof}
	
	%Similarly $Cent_{X \times Y}(A\times B)= \Hat{A} \times Cent_{Y}(B)$ if $rad_{X}(A)> rad_{Y}(B)$ where $\Hat{A}=\{a \in A| d_{X}(a,\partial_{X}(A)) \geq rad_{Y}(B)\} \subseteq A.$
	
	\begin{remark}
		{\normalfont In the above Theorem,\begin{itemize}[noitemsep,nolistsep]
				\item[($i$)]if $rad_{X}(A) = rad_{Y}(B)$, then $\Hat{B}= Cent_{Y}(B).$ So, $Cent_{X \times Y}(A\times B)= Cent_{X}(A) \times Cent_{Y}(B)$.
				\item[($ii$)] if $rad_{X}(A)=0$, then $\Hat{B}=B$ and $Cent_{X}(A)=A$. So, $Cent_{X \times Y}(A\times B)=A\times B.$
				\item [($iii$)] if $rad_{Y}(B)\leq rad_{X}(A)$,  then $Cent_{X\times Y}(A\times B)= \Hat{A} \times Cent_{Y}(B)$, where $\Hat{A}=\{a \in A| d_{X}(a,\partial_{X}(A)) \geq rad_{Y}(B)\}$.
		\end{itemize}} 
	\end{remark}

	From above results, we get
	% can easily determine $rad_{X\times Y}(A\times B).$
	\begin{corollary}  Let $A$ and $B$ are subsets of  metric spaces $X$ and $Y$, respectively.  If the radii of $A$ and $B$ are either both finite or both infinite, then
		$rad_{X \times Y}(A \times B)=\min\{rad_{X}(A),rad_{Y}(B)\}.$
	
	\end{corollary}
		Moreover, if $rad_{X}(A)<Srad_{Y}(B)$ or $rad_{Y}(B)<Srad_{X}(A)$, then the above result is also true. 
%	Generally, it is not easy to see the center and radius of some subsets. In the following examples we find the center and radius of some subsets of Euclidean spaces with the help of Theorem \ref{3.43}
	
	\begin{example}
		{\normalfont Let $A$ and $B$ be subsets of the Euclidean space $\mathbb{R}$. \begin{itemize}
				\item[($i$)] If $A = B= [0,1]$ then $rad_{\mathbb{R}}(A)= rad_{\mathbb{R}}(B)$ and $\Hat{B}= Cent_{Y}(B)=\{\frac{1}{2}\}$. By Theorem \ref{3.43}, $Cent_{\mathbb{R}^{2}}(A\times B)=\{(\frac{1}{2},\frac{1}{2})\}$, and $rad_{\mathbb{R}^{2}}(A\times B)=\frac{1}{2}$.
				\item[($ii$)] If $A=[0,1]$ and $B=[0,5]$ then $rad_{\mathbb{R}}(A)\leq rad_{\mathbb{R}}(B)$ and $\Hat{B}= [0.5,4.5]$. Hence,  $Cent_{\mathbb{R}^{2}}(A\times B)= \{\frac{1}{2}\}\times[0.5,4.5]$ and $rad_{\mathbb{R}^{2}}(A\times B)=\frac{1}{2}$.
				\item[($iii$)] If $A=[0,1]$ and $B=[0,1]\cup [2,4]$ then $rad_{\mathbb{R}}(A)\leq rad_{\mathbb{R}}(B)$ and $\Hat{B}= \{\frac{1}{2}\}\cup [2.5,3.5]$. Hence, $Cent_{\mathbb{R}^{2}}(A\times B)= \{(\frac{1}{2},\frac{1}{2})\}\cup(\{\frac{1}{2}\}\times[2.5,3.5])$ and $rad_{\mathbb{R}^{2}}(A\times B)=\frac{1}{2}$.
		\end{itemize}}
	\end{example}
	
	%\includegraphics[scale=0.75]{Example.png}
	
	%\begin{corollary}
	% Let $(X,d_{X})$ and $(Y,d_{Y})$ be two metric spaces and let $A \subseteq X$ and $B \subseteq Y.$ If $rad_{X}(A)$ or $rad_{Y}(B)$ is zero, then $Cent_{X \times Y}(A\times B)= A\times B.$
	%\end{corollary}
	%\begin{proof}
	% If $rad_{X}(A)=0$ then $\Hat{B}= B$ and $Cent_{X}(A)=A$ (by Theorem\ref{C}). So, $Cent_{X \times Y}(A\times B)= A \times B.$
	%\end{proof}
	
	\begin{example}\label{4.7.}
		{\normalfont Let $\mathbb{R}^{2}$ and $\mathbb{R}$ be the Euclidean spaces. By Theorem \ref{3.43}, we have
			\begin{itemize}
				\item[($i$)] 
				%As $rad_{\mathbb{R}^2}(\mathbb{S}^1)=0,$
				 the center and radius of the cylinder $\mathbb{S}^1\times I$ are $\mathbb{S}^1\times I$ and 0, respectively, where $I=[0,1]$.
				%If $A=\mathbb{S}^{1}\subseteq\mathbb{R}^{2}, B=[0,1]\subseteq\mathbb{R}$ then $A\times B=\mathbb{S}^{1}\times [0,1]\subseteq \mathbb{R}^{3}$ is a cylinder in $\mathbb{R}^{3}$ and $Cent_{\mathbb{R}^{3}}(A\times B)=(A\times B)$ and $rad_{\mathbb{R}^{3}}(A\times B)=0$.
				\item[($ii$)] for $\mathbb{D}^{2}\subseteq\mathbb{R}^{2}$ and $I\subseteq\mathbb{R}, rad_{\mathbb{R}^2}(\mathbb{D}^2)>rad_{\mathbb{R}}(I)$ and $\Hat{\mathbb{D}}^2=\{a\in \mathbb{D}^2|d_{\mathbb{R}^2}(a,\mathbb{S}^1)\geq \frac{1}{2}\}=\{a\in\mathbb{D}^2| |a|\leq\frac{1}{2}\}.$ So, $Cent_{\mathbb{R}^{3}}(\mathbb{D}^2\times I)=\{(a,\frac{1}{2})\in (\mathbb{D}^2\times I)|\,|a|\leq \frac{1}{2}\}$ and $rad_{\mathbb{R}^{3}}(\mathbb{D}^2\times I)=\frac{1}{2}$.
				\item[($iii$)] for $A=\mathbb{D}^{2}\subseteq\mathbb{R}^{2}, B=[0,4]\subseteq\mathbb{R},rad_{\mathbb{R}^2}(A)< rad_{\mathbb{R}}(B)$ and $\Hat{B}=[1,3]$. So, $Cent_{\mathbb{R}^{3}}(A\times B)=
				\{(0,0,b)|1\leq b\leq 3\}$ 
				and $rad_{\mathbb{R}^{3}}(A\times B)=1$.\\
		\end{itemize}}
	\end{example}
	\hspace{5mm}

	Next, we generalize Theorem \ref{3.43} for a finite product $\prod\limits_{i=1}^{n}A_{i}$ of subsets $A_{i}$ of metric spaces $X_{i}, 1\leq i\leq n.$
	 
	\begin{theorem}
		Let $(X_{i},d_{i})$ be metric spaces $,1\leq i\leq n,$ where $n\in \mathbb{N}.$ For $A_{i}\subseteq X_{i},1\leq i\leq n,$ let $\Hat{A_i}=\{a\in A_i|d_i(a,\partial_{X_i}(A_i))\geq\min\{rad_{X_j}(A_j)|1\leq j\leq n\}\}.$ Then $Cent_{\prod\limits_{1}^{n}X_i}(\prod\limits_{1}^{n}A_i)=\prod\limits_{1}^{n}\Hat{A}_i$. Moreover, if  the radii of $A_i$ are either all finite or all infinite, then  $rad_{\prod\limits_{1}^{n}X_i}(\prod\limits_{1}^{n}A_i)=\min\{rad_{X_i}(A_i)|1\leq i\leq n\}.$   
	\end{theorem}
	
	\begin{proof}We prove by induction. If $n=2$ and $rad_{X_1}(A_1)\leq rad_{X_2}(A_2)$ then $\Hat{A}_1=Cent_{X_1}(A_1)$ and  it is true by Theorem \ref{3.43}.\par 
       Assume that it is true for some $k\in \mathbb{N}$. Let $B= \prod\limits_{1}^{k}A_i$. By Induction hypothesis, we have $Cent_{\prod\limits_{1}^{k}X_i}(B)=\prod\limits_{1}^{k}\Hat{A}_i$ and $rad_{\prod\limits_{1}^{k}X_i}(B)=\min\{rad_{X_i}(A_i)|1\leq i\leq k\}.$\par 
       Now, we prove it for $k+1$.\par	
       	%Consider, $\prod\limits_{1}^{k+1}A_i=B\times A_{k+1}$.
		 If $rad_{\prod\limits_{1}^{k}X_i}(B)\leq rad_{X_{k+1}}(A_{k+1}),$ then by Theorem \ref{3.43}, we have $Cent_{\prod\limits_{1}^{k+1}X_i}(\prod\limits_{1}^{k+1}A_i) =Cent_{\prod\limits_{1}^{k+1}X_i}(B\times A_{k+1})= (Cent_{\prod\limits_{1}^{k}X_i}B)\times \Hat{A}_{k+1}= %(\prod\limits_{1}^{k}\Hat{A}_i) \times \Hat{A}_{k+1} =
		 \prod\limits_{1}^{k+1}\Hat{A}_i ,$ and $rad_{\prod\limits_{1}^{k+1}X_i}(\prod\limits_{1}^{k+1}A_i) =rad_{\prod\limits_{1}^{k+1}X_i}(B\times A_{k+1})= rad_{\prod\limits_{1}^{k}X_i}(B)=\min\{rad_{X_i}(A_i)|1\leq i\leq k\}=\min\{rad_{X_i}(A_i)|1\leq i\leq k+1\}$. \par 
		If $rad_{X_{k+1}}(A_{k+1})\leq rad_{\prod\limits_{1}^{k}X_i}(B)$, then by Theorem \ref{3.43}, we have $Cent_{\prod\limits_{1}^{k+1}X_i}(\prod\limits_{1}^{k+1}A_i) %=Cent_{\prod\limits_{1}^{k+1}X_i}(B\times A_{k+1})
		= \Hat{B}\times Cent_{X_{k+1}}(A_{k+1})=  \Hat{B}\times \Hat{A}_{k+1},$ and  $rad_{\prod\limits_{1}^{k+1}X_i}(\prod\limits_{1}^{k+1}A_i) =rad_{\prod\limits_{1}^{k+1}X_i}(B\times A_{k+1})= rad_{X_{k+1}}(A_{k+1})=\min\{rad_{X_i}(A_i)|1\leq i\leq k+1\}.$ Next, we observe that  $\Hat{B}=\prod\limits_{1}^{k}\Hat{A}_i$.
		%$=\prod\limits_{1}^{k}\{a\in A_i|d_i(a,\partial_{X_i}(A_i))\geq rad_{X_i+1}(A_{i+1})\}$.
		We have $\Hat{B}=\{b\in B|d(b,\partial_{\prod\limits_{1}^{k}X_i}(B))\geq rad_{X_{k+1}}(A_{k+1})\}.$ It is easy to see that $d(a,\overline{A}_1\times \overline{A}_2\times...\partial_{X_i}(A_i)\times...\times \overline{A}_k)=d_i(a_i,\partial_{X_i}(A_i)),\forall a=(a_1,a_2,...,a_k)\in \prod\limits_{1}^{k}A_i$. Note that  $d(b,\partial_{\prod\limits_{1}^{k}X_i}(B))=d(b,\bigcup\limits_{i=1}^{k}(\overline{A}_1\times \overline{A}_2\times...\times\partial_{X_i}(A_i)\times...\times\overline{A}_k))=\min\{d(b,(\overline{A}_1\times \overline{A}_2\times...\times\partial_{X_i}(A_i)\times...\overline{A}_k))|1\leq i\leq k\}=\min\{d_i(b,\partial_{X_i}(A_{i}))|1\leq i\leq k\},\forall b\in B.$
		Now, $b=(b_1,b_2,...,b_k)\in \Hat{B} \iff d(b,\partial_{\prod\limits_{1}^{k}X_i}(B))\geq rad_{X_{k+1}}(A_{k+1}) \iff \min\{d_i(b_i,\partial_{X_i}(A_{i}))|1\leq i\leq k\} \geq rad_{X_{k+1}}(A_{k+1}) \iff d_i(b_i,\partial_{X_i}(A_{i})) \geq rad_{X_{k+1}}(A_{k+1}),\\  1\leq i\leq k \iff b_i\in \Hat{A_i}, 1\leq i\leq k \iff b=(b_1,b_2,...,b_k)\in \prod\limits_{1}^{k}\Hat{A_i}.$\\ Thus, $\Hat{B}=\prod\limits_{1}^{k}\Hat{A}_i$. % And, $\Hat{A}_{k+1}=\{a\in A_{k+1}|d_{k+1}(a,\partial_{k+1}(A_{k+1}))\geq rad_{X_{k+1}}(A_{k+1})\}=Cent_{X_{k+1}}(A_{k+1})$.
		 So, $Cent_{\prod\limits_{1}^{k+1}X_i}(\prod\limits_{1}^{k+1}A_i) = %\Hat{B}\times Cent_{X_{k+1}}(A_{k+1})=
		 \prod\limits_{1}^{k}\Hat{A}_i\times \Hat{A}_{k+1}=\prod\limits_{1}^{k+1}\Hat{A}_i.$ Hence, our claim.
	\end{proof}
%	If $n=2$ and $rad_{X_1}(A_1)\leq rad_{X_2}(A_2)$ then $\Hat{A}_1=Cent_{X_1}(A_1).$ So, for $n=2,$ it is true by Theorem \ref{3.43}, and above corollary follows by the induction. 
	
\section{\textbf{Center and Radius of a finite union of subsets of a metric space}}
We know that \cite{Robert} if $A$ and $B$ are subsets of a metric space $(X,d_X)$ such that $A \cap B = \emptyset$, then we have $diam_{X}(A \cup B) \leq diam_{X}(A) + diam_{X}(B) + d_X(A,B)$, where $diam_{X}(A)$ denotes the diameter of $A$.\par 
 In this section, we determine $Cent_{X}(A \cup B)$ and  $rad_{X}(A \cup B)$ for nonclopen subsets $A$ and $B$ of a metric space $X$. \par 
%\begin{itemize}[noitemsep,nolistsep]
		%{\item[(a)] If $A \cap B \neq \phi$, then $diam_{X}(A \cup B) \leq diam_{X}(A) + diam_{X}(B)$, and 
		%	\item[(b)] If $A \cap B = \phi$, then $diam_{X}(A \cup B) \leq diam_{X}(A) + diam_{X}(B) + d(A,B)$ .}\\
%\end{itemize}
For nonclopen subsets $A$ and $B$ of a metric space $X$, let\par 
\hspace{2mm} $\Tilde{A}=\{a\in Cent_{X}(A)|
d_{X}(a, \partial_{X}(B))<rad_{X}(A)\}$, and\par 
\hspace{2mm} $\Tilde{B}=\{b\in Cent_{X}(B)|d_{X}(b,\partial_{X}(A))<rad_{X}(B)\}.$\\
Using these notations, we have the following results:

\begin{theorem}\label{3.31}
	Let A and B be  nonclopen separated subsets of a metric space $(X,d_X).$ Then,
	\begin{itemize}
		\item[(i)]\label{3.33(i)} if $rad_{X}(A)>rad_{X}(B)$ and $Cent_{X}(A)\backslash\Tilde{A}\neq\emptyset,$ then $Cent_{X}(A\cup B)=Cent_{X}(A)\backslash\Tilde{A}$ $\&$ $rad_{X}(A\cup B)=rad_{X}(A),$ and
		\item[(ii)]\label{3.33(ii)} if $rad_{X}(A)=rad_{X}(B)$ and $(Cent_{X}(A)\backslash\Tilde{A})\cup (Cent_{X}(B)\backslash\Tilde{B})\neq\emptyset,$ then $Cent_{X}(A\cup B)=(Cent_{X}(A)\backslash\Tilde{A})\cup (Cent_{X}(B)\backslash\Tilde{B})$ and  $rad_{X}(A\cup B)=rad_{X}(A)=rad_{X}(B).$
	\end{itemize} 
\end{theorem}
\begin{proof} ($i$) % By Theorem \ref{3.31}$(i)$, we have $rad_{X}(A\cup B)=rad_{X}(A).$ 
	For $a\in Cent_{X}(A)\backslash\Tilde{A},$ we have $d_{X}(a,\partial_{X}(B))\geq rad_{X}(A).$ Consequently, $d_{X}(a,\partial_{X}(A\cup B))= \min\{d_{X}(a,\partial_{X}(A)),d_{X}(a,\partial_{X}(B))\}= d_{X}(a,\partial_{X}(A))=rad_{X}(A).$ If $a\notin Cent_{X}(A)\backslash\Tilde{A}$, then either $a\in \Tilde{A}$ or $a\notin Cent_{X}(A)$. If $a\in \Tilde{A},$ then $d_{X}(a,\partial_{X}(B))<rad_{X}(A).$ So, $d_{X}(a,\partial_{X}(A\cup B))\leq d_{X}(a,\partial_{X}(B))<rad_{X}(A).$ If $a\in A$ such that $a\notin Cent_{X}(A)$ then $d_{X}(a,\partial_{X}(A)) <rad_{X}(A)$. So, $d_{X}(a,\partial_{X}(A\cup B))\leq d_{X}(a,\partial_{X}(A))<rad_{X}(A).$ And if $a\in B$ then $d_{X}(a,\partial_{X}(A\cup B))\leq d_{X}(a,\partial_{X}(B))\leq rad_{X}(B)< rad_{X}(A).$ So, for $a\notin Cent_{X}(A)\backslash\Tilde{A},$ we have $d_{X}(a,\partial_{X}(A\cup B))< rad_{X}(A).$ Thus, $Cent_{X}(A\cup B)=Cent_{X}(A)\backslash\Tilde{A}$ and $rad_{X}(A\cup B)=rad_{X}(A).$\\
	$(ii)$ % By Theorem \ref{3.31}$(ii)$, we get $rad_{X}(A\cup B)=rad_{X}(A)=rad_{X}(B).$ It is easy to see that,
	Similarly, for $a\in (Cent_{X}(A)\backslash \Tilde{A})\cup(Cent_{X}(B)\backslash \Tilde{B}),$ we get $d_{X}(a,\partial_{X}(A\cup B))=rad_{X}(A).$ And, for $a\notin (Cent_{X}(A)\backslash \Tilde{A})\cup(Cent_{X}(B)\backslash \Tilde{B}),$ we get $d_{X}(a,\partial_{X}(A\cup B))<rad_{X}(A).$ Thus, $Cent_{X}(A\cup B)=(Cent_{X}(A)\backslash\Tilde{A}) \cup (Cent_{X}(B)\backslash\Tilde{B})$ and $rad_{X}(A\cup B)=rad_{X}(A).$
\end{proof}

Next, we derive relationship of $Srad_{X}(A\cup B)$ with $rad_{X}(A)$ $\&$ $rad_{X}(B).$

\begin{theorem}\label{5.1.}
	Let A and B be  nonclopen separated subsets of a metric space $(X,d_X).$ %such that  $\overline{A} \cap \overline{B} = \emptyset.$
	 Then $Srad_{X}(A\cup B)\leq \max\{rad_{X}(A),rad_{X}(B)\}.$ 
	 %Moreover, if $Cent_{X}(A\cup B)\neq\emptyset$, then $rad_{X}(A\cup B)\leq \max\{rad_{X}(A),rad_{X}(B)\}.$
	%\begin{itemize}
	%	\item[(i)]\label{3.31(i)} if $rad_{X}(A)>rad_{X}(B)$ and $Cent_{X}(A)\backslash\Tilde{A}\neq\emptyset,$ then $rad_{X}(A\cup B)=rad_{X}(A),$
		%\item[(ii)]\label{3.31(ii)} if $rad_{X}(A)>rad_{X}(B)$ and $Cent_{X}(A)\backslash\Tilde{A}=\emptyset,$ then $rad_{X}(A\cup B)<rad_{X}(A),$
	%	\item[(ii)]\label{3.31(iii)} if 
	%	$rad_{X}(A)=rad_{X}(B)$ and $(Cent_{X}(A)\backslash\Tilde{A})\cup (Cent_{X}(B)\backslash\Tilde{B})\neq\emptyset,$ then $rad_{X}(A\cup B)=rad_{X}(A)=rad_{X}(B).$ 
		%\item[(iv)]\label{3.31(iv)}  if $rad_{X}(A)=rad_{X}(B)$ and $(Cent_{X}(A)\backslash\Tilde{A})\cup (Cent_{X}(B)\backslash\Tilde{B})=\emptyset,$ then $rad_{X}(A\cup B)< rad_{X}(A)=rad_{X}(B).$ 
%	\end{itemize}
\end{theorem}
\begin{proof} As $A$ and $B$ are separated, 
	%Note that 	if 	$\overline{A} \cap \overline{B} = \emptyset,$ then
	 $ \partial_{X}(A \cup B)= \partial_{X}(A) \cup \partial_{X}(B)$. So, for $a \in A,$ we have $d_{X}(a,\partial_{X}(A\cup B))\leq d_{X}(a,\partial_{X}(A))\leq rad_{X}(A).$ Similarly, for $b \in B,$ we have $ d_{X}(b,\partial_{X}(A\cup B))\leq d_{X}(b,\partial_{X}(B))\leq rad_{X}(B).$ This implies that $Srad_{X}(A\cup B)\leq \max\{rad_{X}(A),rad_{X}(B)\}.$
\end{proof}

% determine $rad_{X}(A \cup B)$ for nonclopen subsets $A$ and $B$ of a metric space $X$.\\
%If $Cent_{X}(A\cup B)=\emptyset$, then $rad_{X}(A\cup B)$ may be greater than the maximum of these radii.

\begin{theorem}\label{4.2}
		Let A and B be  nonclopen separated subsets of a metric space $(X,d_X).$ Then,
		% If $rad_{X}(A)$ is finite and 
	\begin{itemize}
		\item[(i)]\label{3.31(ii)} if $rad_{X}(B)<rad_{X}(A)<\infty$ and $Cent_{X}(A)\backslash\Tilde{A}=\emptyset,$ then $Srad_{X}(A\cup B)<rad_{X}(A),$ and
		\item[(ii)]\label{3.31(iv)}  if $rad_{X}(A)=rad_{X}(B)<\infty$ and $(Cent_{X}(A)\backslash\Tilde{A})\cup (Cent_{X}(B)\backslash\Tilde{B})=\emptyset,$ then $Srad_{X}(A\cup B)< rad_{X}(A)=rad_{X}(B).$ 
	\end{itemize} %Moeover, if $rad_{X}(A)=\infty$ then $rad_{X}(A\cup B)= rad_{X}(A) =\infty$.
	
\end{theorem}
\begin{proof}	By Theorem \ref{5.1.}, we get $Srad_{X}(A\cup B)\leq rad_{X}(A).$\\ First, let $rad_{X}(A)>rad_{X}(B).$ 
	 As $\Tilde{A}\subseteq Cent_{X}(A)$ and $Cent_{X}(A)\backslash\Tilde{A}= \emptyset$, we get $Cent_{X}(A)=\Tilde{A}.$ As $rad_{X}(A)$ is finite, by Theorem \ref{3.6}, we get $Cent_{X}(A)\neq\emptyset$. So, for $a\in Cent_{X}(A),$ we get $d_{X}(a,\partial_{X}(B))
	<rad_{X}(A).$ Consequently, $d_{X}(a,\partial_{X}(A\cup B))<rad_{X}(A).$ And for $a \in A$ such that $a \notin Cent_{X}(A),$ we get $ d_{X}(a,\partial_{X}(A\cup B)) < rad_{X}(A).$ For $b \in B,$ we have $d_{X}(b,\partial_{X}(A\cup B))\leq rad_{X}(B)$. Therefore, $Srad_{X}(A\cup B)< rad_{X}(A)$.% Also, if $rad_{X}(A)$ is infinite then $Cent_{X}(A)=\emptyset,$ and similarly, as above $rad_{X}(A\cup B)< rad_{X}(A).$
	\par 
	Now, let $rad_{X}(A)=rad_{X}(B).$ We must have
	 %If $rad_{X}(A)$ is finite then
	  both $Cent_{X}(A)$ and $Cent_{X}(B)$ are nonempty. As  $(Cent_{X}(A)\backslash \Tilde{A})\cup (Cent_{X}(B)\backslash \Tilde{B})$ is empty, then $Cent_{X}(A)=\Tilde{A}$ $\&$ $Cent_{X}(B)=\Tilde{B}$. So, for $a \in Cent_{X}(A)$ $\cup \, Cent_{X}(B),$ we get $d_{X}(a,\partial_{X}(A\cup B))< rad_{X}(A).$ Also, for $a\in A\cup B$ such that $a\notin Cent_{X}(A)\cup Cent_{X}(B),$ we have $d_{X}(a,\partial_{X}(A\cup B))< rad_{X}(A).$ Therefore, $Srad_{X}(A\cup B)<rad_{X}(A)=rad_{X}(B).$ %Also, if $rad_{X}(A)$ is infinite then $Cent_{X}(A)\cup Cent_{X}(B)=\emptyset,$ and similarly, as above $rad_{X}(A\cup B)< rad_{X}(A).$
\end{proof}

Above result may not hold if $rad_{X}(A)$ is infinite. For example: Let $A=[2,\infty)$ and $B=[0,1]$ be subsets of Euclidean line $\mathbb{R}.$ Here, $rad_{X}(B)=0.5<rad_{X}(A)=\infty$ and $Cent_{X}(A)\backslash\Tilde{A}=\emptyset$. 
But $Srad_{X}(A\cup B)=\infty\nless rad_{X}(A)$.\par 
% In Example \ref{2.6}, if $B=[4,6]$ then $rad_{X}(B)=1<rad_{X}(A)$ and $Cent_{X}(A)\backslash\Tilde{A}=\emptyset$. As  $Cent_{X}(A\cup B)=\emptyset$, we get $rad_{X}(A\cup B)=\infty$. %The following result deals with the case when $rad_{X}(A)=\infty$.

Notice that, in the above theorems, if $Cent_{X}(A\cup B)\neq\emptyset$, then $Srad_{X}(A\cup B)$ can be replaced with $rad_{X}(A\cup B)$. On the other hand, if $Cent_{X}(A\cup B)=\emptyset$, then above results may not hold by replacing $Srad_{X}(A)$ with $rad_{X}(A)$. For example: Consider, a metric subspace $X=\mathbb{R}\times (\{0\}\cup[1,\infty))$ of Euclidean space $\mathbb{R}^2$. Let $A=\bigcup_{n\in\mathbb{N}}[10n,10n+5]\times\{0\}$ and $B=\bigcup_{n\in\mathbb{N}}[10n,10n+5]\times\{2-\frac{1}{n}\} $ be subsets of $X$. Here, $Cent_{X}(A)=\Tilde{A}=\bigcup_{n\in\mathbb{N}}\{10n+2.5\}$, $rad_{X}(A)=2.5$, $rad_{X}(B)=0$ and $Srad_{X}(A\cup B)=2$. We can also notice that $Cent_{X}(A\cup B)=\emptyset$ and $rad_{X}(A\cup B)=\infty \nleq rad_{X}(A)$.
%\leq \max\{rad_{X}(A),rad_{X}(B)\}.$

\begin{remark}\label{3.32}
	{\normalfont In Theorem \ref{4.2}(i), %if $rad_{X}(A)>rad_{X}(B)$ and $Cent_{X}(A)\backslash\Tilde{A}=\emptyset,$ then $rad_{X}(A\cup B)<rad_{X}(A).$ W
		we further establish a relationship between $Srad_{X}(A\cup B)$ and $rad_{X}(B)$. Define $\Tilde{\Tilde{A}}=\{a \in A| d_{X}(a,\partial_{X}(A\cup B))>rad_{X}(B)\}$. Then \begin{itemize}[noitemsep,nolistsep]
			\item[(i)] if $\Tilde{\Tilde{A}}\neq \emptyset,$ then $rad_{X}(B)< Srad_{X}(A\cup B)$, and    \item[(ii)] if $\Tilde{\Tilde{A}}=\emptyset$ then $Srad_{X}(A\cup B)\leq rad_{X}(B).$
	\end{itemize} }
\end{remark}

Note that for any point $a\in A\cup B$ such that $a\notin \Tilde{\Tilde{A}},$ we get $d_X(a,\partial_{X}(A\cup B))\leq rad_{X}(B).$ It is easy to observe that  if $\Tilde{\Tilde{A}}\neq\emptyset$ then all those points of $A\cup B$ which are at the maximum distance from $\partial_{X}(A\cup B)$ are in $\Tilde{\Tilde{A}}.$ So, we get
\begin{remark}
	{\normalfont   Let A and B are  nonclopen separated subsets of a metric space $(X,d_X)$. If $\Tilde{\Tilde{A}}\neq\emptyset$, then $Cent_{X}(A\cup B)\subseteq \Tilde{\Tilde{A}}$.
		%	\begin{itemize}
			%		\item [($i$)] if $rad_{X}(A)>rad_{X}(B)$ such that $ (Cent_{X}(A)\backslash\Tilde{A})=\emptyset$ and $\Tilde{\Tilde{A}}\neq\emptyset$ then $Cent_{X}(A\cup B)\subseteq \Tilde{\Tilde{A}}$, and
			%		\item [($ii$)]  if $rad_{X}(A)=rad_{X}(B)$ such that $ (Cent_{X}(A)\backslash\Tilde{A})\cup (Cent_{X}(B)\backslash\Tilde{B})=\emptyset$ and $\Tilde{\Tilde{A}}\cup \Tilde{\Tilde{B}}\neq\emptyset,$ then $Cent_{X}(A\cup B)\subseteq \Tilde{\Tilde{A}}\cup\Tilde{\Tilde{B}}$, where $\Tilde{\Tilde{B}}=\{b \in B| d_{X}(b,\partial_{X}(A\cup B))>rad_{X}(A)\}$.
			%	\end{itemize}
	}
\end{remark}

\begin{theorem}\label{4.3}
	Let A and B be  nonclopen separated subsets of a metric space $(X,d_X)$, such that $rad_{X}(B)<rad_{X}(A)=\infty$ and $rad_{X}(B)\geq Srad_{X}(A)$. Then,
	\begin{itemize}
	%	\item [(i)] if $rad_{X}(B)$ is finite such that $rad_{X}(B)<Srad_{X}(A)$, then $rad_{X}(A\cup B)=\infty$.
		\item [(i)]  if 
		%$rad_{X}(B)$ is finite such that $rad_{X}(B)\geq Srad_{X}(A)$ and 
		$Cent_{X}(B)\backslash\Tilde{B}\neq\emptyset$, then $rad_{X}(A\cup B)=Srad_{X}(A\cup B)=rad_{X}(B)$ $\&$ \\$Cent_{X}(A\cup B)=Cent_{X}(B)\backslash\Tilde{B}$, and
		\item [(ii)]  if 
		%$rad_{X}(B)$ is finite such that $rad_{X}(B)\geq Srad_{X}(A)$ and
		 $Cent_{X}(B)\backslash\Tilde{B}=\emptyset$, then $Srad_{X}(A\cup B)<rad_{X}(B)$.
		%\item [(iv)] if $rad_{X}(B)=\infty$,  then $rad_{X}(A\cup B)=\infty$.
	\end{itemize}
\end{theorem}

\begin{proof}
	As $rad_{X}(A)=\infty$, by Theorem  \ref{3.6}, $Cent_{X}(A)=\emptyset$, which means $ d_{X}(a,\partial_{X}(A))\\ <Srad_{X}(A) ,\forall a\in A.$ Since  $rad_{X}(B)$ is finite, we get $Cent_{X}(B)\neq\emptyset$.\\
	% $(i)$ As $d_{X}(a,\partial_{X}(B))\leq rad_{X}(B) < Srad_{X}(A), \forall a\in B.$ So, $d_{X}(a,\partial_{X}(A\cup B))=\min\{d_{X}(a,\partial_{X}(A)),d_{X}(a,\partial_{X}(B))\}< Srad_{X}(A), \forall a\in A\cup B.$ And, for every $a\in A$, there exist $b \in A$ such that $d_{X}(a,\partial_{X}(A\cup B))< d_{X}(b,\partial_{X}(A\cup B))$. That implies $Cent_{X}(A\cup B)=\emptyset$. Hence $rad_{X}(A\cup B)=\infty$.\\
	 $(i)$ If $a\in Cent_{X}(B)\backslash\Tilde{B}$, then  $a \in Cent_{X}(B)$ such that $d_{X}(a,\partial_{X}(A))\geq rad_{X}(B).$ Thus, $d_{X}(a,\partial_{X}(A\cup B))=\min\{d_{X}(a,\partial_{X}(A)),d_{X}(a,\partial_{X}(B))\}=d_{X}(a,\partial_{X}(B))$ $=rad_{X}(B).$ And, if $a\notin Cent_{X}(B)\backslash\Tilde{B}$, then $d_{X}(a,\partial_{X}(A\cup B))< rad_{X}(B).$ So, $Cent_{X}(A\cup B)=Cent_{X}(B)\backslash\Tilde{B}$ and  $rad_{X}(A\cup B)=Srad_{X}(A\cup B)=rad_{X}(B).$\\
	 $(ii)$ Similarly, if $Cent_{X}(B)\backslash\Tilde{B}=\emptyset$, then  $Srad_{X}(A\cup B)< rad_{X}(B).$
	% $(iv)$ If $rad_{X}(B)=\infty$, then $Cent_{X}(B)=\emptyset$. So, $ d_{X}(a,\partial_{X}(B)) <\sup_{a' \in B}d_{X}(a',\partial_{X}(B)),\\ \forall a\in B.$ We have $d_{X}(a,\partial_{X}(A\cup B))%=\min\{d_{X}(a,\partial_{X}(A)),d_{X}(a,\partial_{X}(B))\}
%	 < \sup_{a' \in A}d_{X}(a',\partial_{X}(A)), \forall a\in A,$ and $d_{X}(a,\partial_{X}(A\cup B)< \sup_{a' \in B}d_{X}(a',\partial_{X}(B)), \forall a\in B.$ So, we get $d_{X}(a,\partial_{X}(A\cup B)< \max\{ \sup_{a' \in A}\\ d_{X}(a',\partial_{X}(A)),\sup_{a' \in B}d_{X}(a',\partial_{X}(B))\}, \forall a\in A\cup B.$ And for every $a\in A\cup B$, there exist $b\in A\cup B$ such that $d_{X}(a,\partial_{X}(A\cup B))< d_{X}(b,\partial_{X}(A\cup B)).$ So, $Cent_{X}(A\cup B)=\emptyset$ and hence $rad_{X}(A)=\infty$.
\end{proof}
%\begin{remark}
  %{\normalfont Above theorems hold even if we replace the condition 	$\overline{A} \cap \overline{B}=\emptyset$ with a weaker condition 	$\overline{A} \cap B=\emptyset$ and 	$A \cap \overline{B}=\emptyset$ in Theorem \ref{3.31}.}
%\end{remark}

%\begin{remark}
%	{\normalfont In Theorems \ref{3.31} and \ref{4.2},
%	\begin{itemize}[noitemsep,nolistsep]
%		\item[(i)] if $rad_{X}(A)<rad_{X}(B)$ and $Cent_{X}(B)\backslash\Tilde{B}\neq\emptyset,$ then $rad_{X}(A\cup B)=rad_{X}(B),$ and
%		\item[(ii)] if $rad_{X}(A)<rad_{X}(B)$ and $Cent_{X}(B)\backslash\Tilde{B}=\emptyset,$ then $rad_{X}(A\cup B)<rad_{X}(B).$
%   \end{itemize}}
%\end{remark}

%Now we have a similar result for $Cent_{X}(A\cup B)$ in some cases where $A$ and $B$ are nonclopen subsets of metric space $X.$

%\begin{theorem}\label{3.33}
%	Let A and B be nonclopen subsets of a metric space $(X,d_X)$ such that $\overline{A} \cap \overline{B} = \emptyset.$ Then, \begin{itemize}
%		\item[(i)]\label{3.33(i)} if $rad_{X}(A)>rad_{X}(B)$ and $Cent_{X}(A)\backslash\Tilde{A}\neq\emptyset,$ then $Cent_{X}(A\cup B)=Cent_{X}(A)\backslash\Tilde{A},$ and
%		\item[(ii)]\label{3.33(ii)} if $rad_{X}(A)=rad_{X}(B)$ and $(Cent_{X}(A)\backslash\Tilde{A})\cup (Cent_{X}(B)\backslash\Tilde{B})\neq\emptyset,$ then $Cent_{X}(A\cup B)=(Cent_{X}(A)\backslash\Tilde{A})\cup (Cent_{X}(B)\backslash\Tilde{B})$.
%	\end{itemize} 
%\end{theorem}

\begin{example}
	{\normalfont Let $X\subseteq \mathbb{R}^2$ be the union of two rectangles with vertices $(0,0),(1,0),\\
		(1,2)$ $\&$ $(0,2)$ and $(1,0),(2,0),(2,2)$ $\&$ $(1,2).$ Thus $X$ is a metric subspace of the Euclidean space $\mathbb{R}^2.$ We consider nonclopen separated subsets $A$ and $B$ of $X$. 
		\begin{itemize}
			\item[(i)] Let $A$ and $B$ be the line segments joining $(0,0)$ to $(0,2)$ and  $(2,1)$ to $(2,2),$ respectively. Here $rad_{X}(A)=1 > 0.5 =rad_{X}(B), Cent_{X}(A)=\{(0,1)\}$ and $\Tilde{A}=\emptyset.$ By Theorem \ref{3.31}$(i)$, we get $rad_{X}(A\cup B)=1$ $\&$ $Cent_{X}(A\cup B)=\{(0,1)\}.$
			\item[(ii)]  Let $A$ and $B$ be the line segments joining $(0,0)$ to $(0,2)$ and  $(2,0)$ to $(2,2),$ respectively. Here $rad_{X}(A)=1 =rad_{X}(B), Cent_{X}(A)=\{(0,1)\}, Cent_{X}(B)\\
			=\{(2,1)\}$ and $\Tilde{A}=\emptyset=\Tilde{B}.$ By Theorem \ref{3.31}$(ii)$, $rad_{X}(A\cup B)=1$ and $Cent_{X}(A\cup B)=\{(0,1),(2,1)\}.$
			\item[(iii)] Let $A$ and $B$ be the line segments joining $(0,0)$ to $(2,0)$ and  $(1,0.2)$ to $(1,1),$ respectively.  Here $rad_{X}(A)=1 > 0.4 =rad_{X}(B), Cent_{X}(A)=\{(1,0)\}$ and $\Tilde{A}=\{(1,0)\}.$ By Theorem \ref{4.2}$(i)$, $Srad_{X}(A\cup B)<1.$ Note that for $a=(0.5,0), d_{X}(a,\partial_{X}(A\cup B))>0.4=rad_{X}(B).$ So, $\Tilde{\Tilde{A}}\neq\emptyset,$ and hence by Remark \ref{3.32}$(i)$, $Srad_{X}(A\cup B)> 0.4$.
			\item[(iv)]  Let $A$ and $B$ be the line segments joining $(0,0)$ to $(2,0)$ and  $(1,0.2)$ to $(1,2),$ respectively. Here $rad_{X}(A)=1 > 0.9 =rad_{X}(B), Cent_{X}(A)=\{(1,0)\}$ and $\Tilde{A}=\{(1,0)\}.$ By Theorem \ref{4.2}$(i)$, $Srad_{X}(A\cup B)<1.$ Infact, as $\Tilde{\Tilde{A}}=\emptyset,$ by Remark \ref{3.32}$(ii)$, $Srad_{X}(A\cup B)\leq 0.9$.
			
	\end{itemize}}
\end{example}
One can easily verify that the radius and center of $A\cup B$ in above all four cases are the same as we have obtained using Theorems \ref{3.31} and \ref{4.2}.
\begin{example}
	{\normalfont    Let $\mathbb{S}^2 \subseteq \mathbb{R}^3$ be the unit sphere with  metric induced from the Euclidean space $\mathbb{R}^3$. Let $A=\{(x,y,z)\in \mathbb{S}^2|y=0\}\backslash B_{\mathbb{S}^2}((1,0,0),0.1)$ and $B=\{(x,y,z)\in \mathbb{S}^2|z=0\}\backslash B_{\mathbb{S}^2}((-1,0,0),0.1)$ be nonclopen subsets of $\mathbb{S}^2$ such that $\overline{A}\cap\overline{B}=\emptyset,$ where $B_{\mathbb{S}^2}((1,0,0),0.1)$ and $B_{\mathbb{S}^2}((-1,0,0),0.1)$ are open balls centred at $(1,0,0)$ and $(-1,0,0)$ respectively, with radius $0.1$. Here $rad_{\mathbb{S}^2}(A)=rad_{\mathbb{S}^2}(B)\approx 1.97$, $Cent_{\mathbb{S}^2}(A)=\{(-1,0,0)\}, Cent_{\mathbb{S}^2}(B)=\{(1,0,0)\}$ and $\Tilde{A}=\{(-1,0,0)\}, \Tilde{B}=\{(1,0,0)\}.$ This implies $Cent_{\mathbb{S}^2}(A)\backslash\Tilde{A}$ and $Cent_{\mathbb{S}^2}(B)\backslash\Tilde{B}$ are empty. By Theorem \ref{4.2}(ii), $Srad_{\mathbb{S}^2}(A\cup B)< 1.97.$}
\end{example}
Next, we generalise above results for a finite union $\bigcup\limits_{i=1}^{n}A_{i}$ of nonclopen subsets $A_{i},1\leq i\leq n,$ of a metric space $X.$ 

\begin{theorem}\label{3.37}
	Let $(X,d_X)$ be a metric space. For nonclopen subsets $A_i \subseteq X, 1\leq i\leq n$, such that $A_i$ $\&$ $A_j$ are separated, %$\overline{A_i}\cap \overline{A_j}=\emptyset,$
	 for all $i\neq j$ and $n\in \mathbb{N},$
	% be a bounded non clopen subset of a metric space X such that $A= A_{1} \cup A_{2} \cup ...\cup A_{n},$ where $A_{i}'s$ are non clopen subsets of $X$ such that $\overline{A_i}\cap \overline{A_j}=\emptyset,\,\forall i,j$ and $n\in\mathbb{N}$.
	 let $\Tilde{A}_{j}=\{a\in Cent_{X}(A_{j})|
	 d_{X}(a,\partial_{X}(A_{i}))<rad_{X}(A_{j}),\text{ for some } i\neq j\}, 1\leq j\leq n$. Let $M$ be the collection of all those $A_j$ such that $rad_{X}(A_{j})= \max\{rad_{X}(A_{i})|1\leq i \leq n\}$ and $Cent_{X}(A_j)\backslash\Tilde{A_j}\neq\emptyset$. Then, $Srad_{X}(\bigcup\limits_{1}^{n}A_{i}) \leq  \max\{rad_{X}(A_{i})|1\leq i\leq n\}$.\\
	  Moreover, if $\bigcup\limits_{A_j\in M}(Cent_{X}(A_j)\backslash\Tilde{A_j})\neq\emptyset,$ 
	 then  $Cent_{X}(\bigcup\limits_{1}^{n}A_{i}) = \bigcup\limits_{A_{j}\in M}(Cent_{X}(A_{j})\backslash \Tilde{A_{j}})$ $\&$ $rad_{X}(\bigcup\limits_{1}^{n}A_{i}) = \max\{rad_{X}(A_{i})|1\leq i \leq n\}.$

	 % \begin{itemize}
	%	\item[(i)] if %$M\neq\emptyset$,
	%	$\bigcup\limits_{A_j\in M}(Cent_{X}(A_j)\backslash\Tilde{A_j})\neq\emptyset,$ 
	%	then  $Cent_{X}(\bigcup\limits_{1}^{n}A_{i}) = \bigcup\limits_{A_{j}\in M}(Cent_{X}(A_{j})\backslash \Tilde{A_{j}})$ $\&$ $rad_{X}(\bigcup\limits_{1}^{n}A_{i}) = \max\{rad_{X}(A_{i})|1\leq i \leq n\},$ and 
	%	\item[(ii)] if %$M=\emptyset$,
	%	$\bigcup\limits_{A_j\in M}(Cent_{X}(A_j)\backslash\Tilde{A_j})=\emptyset$,
	%	 then $rad_{X}(\bigcup\limits_{1}^{n}A_{i}) <  \max\{rad_{X}(A_{i})|1\leq i\leq n\}$.
	% \end{itemize}
\end{theorem} 

\begin{proof}
	We prove it by induction. If $n=2,$ it is true by Theorems \ref{3.31} $\&$  \ref{5.1.}.\\
	Assume that it is true for some $k\in \mathbb{N}$. Let $B=\bigcup\limits_{1}^{k}A_{i}$ and $K$ be the collection of all those $A_j$ such that $rad_{X}(A_{j})= \max\{rad_{X}(A_{i})|1\leq i \leq k\}$ and $Cent_{X}(A_j)\backslash\Tilde{A_j}\neq\emptyset$.\\ By induction hypothesis, we have $Srad_{X}(B) \leq  \max\{rad_{X}(A_{i})|1\leq i\leq k\}$, and  if %$K\neq\emptyset$,
	$\bigcup\limits_{A_j\in K}(Cent_{X}(A_j)\backslash\Tilde{A_j})\neq\emptyset,$
	then $Cent_{X}(B) = \bigcup\limits_{A_{j}\in K}(Cent_{X}(A_{j})\backslash \Tilde{A_{j}})$ and
	$rad_{X}(B)=\max\{rad_{X}(A_{i})
	|1\leq i \leq k\}.$ \\
	%\begin{itemize}
	%	\item[(i)] if %$K\neq\emptyset$,
	%	$\bigcup\limits_{A_j\in K}(Cent_{X}(A_j)\backslash\Tilde{A_j})\neq\emptyset,$
	%	 then $Cent_{X}(B) = \bigcup\limits_{A_{j}\in K}(Cent_{X}(A_{j})\backslash \Tilde{A_{j}})$ and\\
	%	  $rad_{X}(B)=\max\{rad_{X}(A_{i})
		%  |1\leq i \leq k\},$ and 
	%	\item[(ii)] if %$K=\emptyset$,
	%	 $\bigcup\limits_{A_j\in K}(Cent_{X}(A_j)\backslash\Tilde{A_j})=\emptyset$,
	%	 then $rad_{X}(B) <  \max\{rad_{X}(A_{i})|1\leq i\leq k\}$. 
	%\end{itemize}
	Now, we prove it for $k+1$. Let $K'$ be the collection of all those $A_j$ such that $rad_{X}(A_{j})= \max\{rad_{X}(A_{i})|1\leq i \leq k+1\}$ and $Cent_{X}(A_j)\backslash\Tilde{A_j}\neq\emptyset$.\par
	If $rad_{X}(B)<rad_{X}(A_{k+1})$, then by Theorems  \ref{3.31} $\&$  \ref{5.1.}, we have $Srad_{X}(\bigcup\limits_{1}^{k+1}A_{i}) = Srad_{X}(B\cup A_{k+1})\leq rad_{X}(A_{k+1})\leq \max\{rad_{X}(A_{i})|1\leq i\leq k+1\}$. And if\\ $Cent_{X}(A_{k+1})\backslash\Tilde{A}_{k+1}\neq\emptyset,$
	then $Cent_{X}(\bigcup\limits_{1}^{k+1}A_{i})= Cent_{X}(B\cup A_{k+1})= Cent_{X}(A_{k+1})\backslash\Tilde{A}_{k+1}$, and $rad_{X}((\bigcup\limits_{1}^{k+1}A_{i}))=rad_{X}(A_{k+1}) = \max\{rad_{X}(A_{i})|1\leq i \leq k+1\}$.\par
	% \begin{itemize}
	%	\item[(i)] if %$K'\neq \emptyset \implies K'=\{A_{k+1}\}\implies
	%	$Cent_{X}(A_{k+1})\backslash\Tilde{A}_{k+1}\neq\emptyset,$
	%	 then $Cent_{X}(\bigcup\limits_{1}^{k+1}A_{i})= Cent_{X}(B\cup A_{k+1})= Cent_{X}(A_{k+1})\backslash\Tilde{A}_{k+1}$, and $rad_{X}((\bigcup\limits_{1}^{k+1}A_{i}))=rad_{X}(A_{k+1}) = \max\{rad_{X}(A_{i})|1\leq i \leq k+1\},$ and 
	%	\item[(ii)] if %$K'=\emptyset\implies
	%	$Cent_{X}(A_{k+1})\backslash\Tilde{A}_{k+1}=\emptyset,$ then $rad_{X}(\bigcup\limits_{1}^{k+1}A_{i}) < rad_{X}(A_{k+1})= \max\{rad_{X}
	%	(A_{i})|\\
	%	1\leq i\leq k+1\}$. 
	%\end{itemize}\par
	If $rad_{X}(A_{k+1})<rad_{X}(B)$, then again by Theorems  \ref{3.31} $\&$  \ref{5.1.}, we have  $Srad_{X}(\bigcup\limits_{1}^{k+1}A_{i})\\ =Srad_{X}(B\cup A_{k+1})\leq rad_{X}(B)\leq \max\{rad_{X}(A_{i})|1\leq i\leq k+1\}$. And if \\
	%\begin{itemize}
	%	\item[(i)] if %$K'\neq \emptyset\implies 
		$Cent_{X}(B)\backslash\Tilde{B}\neq\emptyset,$ then $Cent_{X}(\bigcup\limits_{1}^{k+1}A_{i})= %Cent_{X}(B\cup A_{k+1}) =
		Cent_{X}(B)\backslash\Tilde{B}$ and $rad_{X}(\bigcup\limits_{1}^{k+1}A_{i})=rad_{X}(B) \\= \max\{rad_{X}(A_{i})|1\leq i \leq k+1\}.$
		
		 Next, we observe that in this case $Cent_{X}(B)\backslash\Tilde{B}= \bigcup\limits_{A_{j}\in K'}(Cent_{X}(A_{j})\backslash \Tilde{A_{j}})$.\\ We have $\Tilde{B}=\{b\in Cent_{X}(B)|d_{X}(b,\partial_{X}(A_{k+1}))< rad_{X}(B)\}$. Let $b\in Cent_{X}(B)\backslash\Tilde{B}.$ Then $b\in Cent_{X}(B)=\bigcup\limits_{A_{j}\in K}(Cent_{X}(A_{j})\backslash \Tilde{A_{j}})$ and $b\notin \Tilde{B}\implies b\in Cent_{X}(A_{j})\backslash \Tilde{A_{j}}$ for some $A_j\in K$ and $d_{X}(b,\partial_{X}(A_{K+1}))\geq rad_{X}(B)=rad_{X}(A_j).$ This gives that $Cent_{X}(A_{j})\backslash \Tilde{A_{j}}\neq\emptyset, \text{ and } rad_{X}(A_j)=\max\{rad_{X}(A_i|1\leq i\leq k+1)\}\implies A_j\in K'\implies b\in \bigcup\limits_{A_{j}\in K'}(Cent_{X}(A_{j})\backslash \Tilde{A_{j}})$.\\ Conversely, let $b\in \bigcup\limits_{A_{j}\in K'}(Cent_{X}(A_{j})\backslash \Tilde{A_{j}})
		  \implies b\in Cent_{X}(A_{j})\backslash \Tilde{A_{j}}$ for some $A_j\in K'\implies b\notin \Tilde{A}_j\implies A_j\in K \text{ and } d_{X}(b,\partial_{X}(A_{k+1}))\geq rad_{X}(A_j)=rad_{X}(B)\implies b\in Cent_{X}(B) \text{ and }
		   b\notin \Tilde{B}\implies b\in Cent_{X}(B)\backslash\Tilde{B}. $\par
		 %=\emptyset.$ Because if $b\in \Tilde{B} \implies b \in Cent_{X}(B)= \bigcup\limits_{A_{j}\in K}(Cent_{X}(A_{j})\backslash \Tilde{A_{j}})$, which means $b\in (Cent_{X}(A_{j})\backslash \Tilde{A_{j}})$ for some $A_j\in K$. That implies $d_{X}(b,\partial_{X}(A_{k+1}))\geq rad_{X}(A_j)=\max\{rad_{X}(A_i)|1\leq i\leq k\}$
	%	\item[(ii)] if %$K'= \emptyset\implies 
		%$Cent_{X}(B)\backslash\Tilde{B}=\emptyset,$ then $rad_{X}(\bigcup\limits_{1}^{k+1}A_{i})<rad_{X}(B) =  \max\{rad_{X}(A_{i})|1\leq i\leq k+1\}$.
%	\end{itemize}\par 
	If $rad_{X}(A_{k+1})=rad_{X}(B)$, then by Theorems  \ref{3.31} $\&$  \ref{5.1.}, we have $Srad_{X}(\bigcup\limits_{1}^{k+1}A_{i}) =Srad_{X}(B\cup A_{k+1})\leq rad_{X}(B)\leq \max\{rad_{X}(A_{i})|1\leq i\leq k+1\}$. And if % \begin{itemize}
%	\item[(i)] if %$K'\neq \emptyset\implies 
	$(Cent_{X}(B)\backslash\Tilde{B})\cup(Cent_{X}(A_{k+1})\backslash\Tilde{A}_{k+1})\neq\emptyset,$ then $Cent_{X}(\bigcup\limits_{1}^{k+1}A_{i})= Cent_{X}(B\cup A_{k+1}) = (Cent_{X}(B)\backslash\Tilde{B})\cup(Cent_{X}(A_{k+1})\backslash\Tilde{A}_{k+1})$ and $rad_{X}(\bigcup\limits_{1}^{k+1}A_{i})=rad_{X}(A_{k+1})
	=rad_{X}(B) = \max\{rad_{X}(A_{i})|1\\
	\leq i \leq k+1\}.$ In this case, it is easy to observe that $(Cent_{X}(B)\backslash\Tilde{B})\cup(Cent_{X}(A_{k+1})\backslash\Tilde{A}_{k+1})\\
	= \bigcup\limits_{A_{j}\in K'}(Cent_{X}(A_{j})\backslash \Tilde{A_{j}}).$
%	\item[(ii)] if %$K'= \emptyset\implies 
%	$(Cent_{X}(B)\backslash\Tilde{B})\cup(Cent_{X}(A_{k+1})\backslash\Tilde{A}_{k+1})=\emptyset,$ then $rad_{X}(\bigcup\limits_{1}^{k+1}A_{i})<rad_{X}(B) =  \max\{rad_{X}(A_{i})|1\leq i\leq n\}$.
%\end{itemize}
 Thus, it is true for $i=k+1$. Hence, our claim.
\end{proof}

\section{\textbf{Largest open balls contained in a subset of metric space}}
 For any  subset $A$ of a metric space $X,$ there is a natural question to identify the %open balls of $X$ which are entirely contained in $A$ with the largest radius.
largest open balls (largest open ball means an open ball of  $X$ with the largest radius) that are entirely contained in $A$. To answer this question, we introduce a notion of quasi-center and quasi-radius of a subset $A$ of metric space $X$:

\begin{definition}[Quasi-center of a subset]
    The quasi-center of $A$ is the set $\{a\in A|d_{X}(a,A^c)\geq d_{X}(b,A^c),\forall b\in A\},$ where $A^c$ denotes the complement of $A$ in $X$. We denote the quasi-center of $A$ in $X$ by $QCent_{X}(A).$\par
    Thus the quasi-center of $A$ is the set of all those elements of $A$ which are at the maximum distance from $A^c.$
\end{definition}

\begin{definition}[Quasi-radius of a subset]
    The quasi-radius of a subset $A$ of metric space $X$ is the distance between its quasi-center and its complement in $X.$ We denote the quasi-radius of $A$ in $X$ by $Qrad_{X}(A)$.\par
    Notice that $Qrad_{X}(A)=d_{X}(QCent_{X}(A),A^c)=d_{X}(a,A^c),\forall a\in QCent_{X}(A).$
\end{definition}

\begin{example}\label{5.3}
	{\normalfont    Let $X \subseteq\mathbb{R}^2$ denote the union of $A$ and $B$, where A is a semi unit circle $\{(x,y)\in \mathbb{R}^2|x^2+y^2=1$ and $ x\geq 0\}$ and $B$ is the union of three line segments joining $(i)$  $(0,1)$ to $(1.5,1)$, $(ii)$ $(1.5,1)$ to $(1.5,-1)$ and  $(iii)$ $(0,-1)$ to $(1.5,-1)$. Consider $X=A\cup B$ as a metric subspace of the Euclidean metric space  $\mathbb{R}^2$.  Here $Cent_{X}(B)=\{(1.5,0)\}$ $\&$ $rad_{X}(B)\approx 1.8$, and  $QCent_{X}(B)=\{(1.5,1),(1.5,-1)\}$ $\&$ $Qrad_{X}(B)\approx 0.8$.}
\end{example}

As the complement of a metric space $X$ is empty in itself, we get $QCent_{X}(X)=X$ $\&$ $Qrad_{X}(X)=\infty$. Also, we have $QCent_{X}(\emptyset)=\emptyset$ $\&$ $Qrad_{X}(\emptyset)=\infty$. \\

As the quasi-center of A  consists all those points of A which are at the maximum distance from its complement, $Qrad_{X}(A)$ is the maximum distance of any point $a \in A$ from its complement. It is clear that if $QCent_X(A)\neq \emptyset$ then $Qrad_{X}(A) = \sup_{a \in A}d_{X}(a,A^c)$. And, if  $QCent_X(A) = \emptyset$, then $Qrad_{X}(A)=\infty$, but $\sup_{a \in A}d_{X}(a,A^c)$ could be finite. \par 
It leads us to introduce the notion of  the semi-quasi-radius of a subset $A$ of a metric space $X$.

\begin{definition}(Semi-quasi-radius)
The semi-quasi-radius of a subset $A$ of metric space $X$ is the supremum of the set that consists of distance of any point $a\in A$ from $A^c$. We denote the semi-quasi-radius of A in X by $SQrad_{X}(A).$\par That is, $SQrad_{X}(A)= \sup_{a \in A}d_{X}(a,A^c)$.
\end{definition}

 Note that, $Qrad_{X}(A)\geq SQrad_{X}(A)$. In Example \ref{3.6.} and \ref{3.7.}, it is easy to see that  $Qrad_{X}(A) > SQrad_{X}(A)$.
\begin{lemma}
	Let $A$ be a subset of metric space $X$. 
	%If $QCent_{X}(A)\neq\emptyset$, 
	Then, $ SQrad_{X}(A)\leq rad_{X}(A).$
\end{lemma}
\begin{proof}
	As $\partial_{X}(A) \subseteq \overline{A^{c}} \implies d_{X}(a,\partial_{X}(A))\geq d_{X}(a,\overline{A^c}) = d_{X}(a,A^c), \forall a \in A$ $\implies$ $rad_{X}(A)\geq \sup_{a'\in A}d_{X}(a',\partial_{X}(A))\geq d_{X}(a,\partial_{X}(A)) \geq d_{X}(a,A^c),\forall a\in A.$  Thus, $SQrad_{X}(A)\leq rad_{X}(A).$
\end{proof}

 If $QCent_{X}(A)\neq\emptyset$, then by above lemma $ Qrad_{X}(A)\leq rad_{X}(A).$ If $QCent_{X}(A)=\emptyset,$ then it may not be true. For Example: Consider, $X=(\mathbb{R}\times\{0\})\cup B$ with subspace metric from Euclidean space $\mathbb{R}^2$, where $B=\bigcup_{n\in\mathbb{N}}[10n,10n+5]\times\{2-\frac{1}{n}\}.$ And, let $A=\bigcup_{n\in\mathbb{N}}[10n,10n+5]\times\{0\}.$ Here, $QCent_{X}(A)=\emptyset$, $SQrad_{X}(A)=2$ but $Qrad_{X}(A)=\infty \nleq rad_{X}(A)=2.5.$

\begin{remark}\label{5.5}
     {\normalfont A metric space $(X,d_X)$ is a path metric space if  the distance between each pair of points equals the infimum of the lengths of the curves joining the points\cite{Gro99}. Recall that the Euclidean spaces and connected Riemannian manifolds are path metric spaces. If $X$ is a path metric space, then for a proper subset $A$ of $X$, we have $d_{X}(a,\partial_{X}(A))= d_{X}(a,A^c),\forall a\in A.$ This gives that for path metric space $X,$ $Cent_{X}(A)= QCent_{X}(A)$, and $rad_{X}(A)= Qrad_{X}(A)$.\par Notice that, in Example \ref{2.3} and \ref{2.4}, quasi-center and quasi-radius of subsets are the same as their center and radius, respectively. By Example \ref{5.3}, we see that the  above remark is not true if $X$ is not a path metric space.
     }
\end{remark}

It is easy to observe the following results:

\begin{lemma}\label{3.16}
   Let $A$ be a nonempty subset of a metric space X such that $ A \subseteq \partial_{X}(A)$. Then, $QCent_{X}(A)= A$ and $Qrad_{X}(A)= 0$.
\end{lemma} 

\begin{lemma}\label{3.17}
    Let $A$ be a subset of metric space $X$ with nonempty interior. Then $QCent_{X}(A)$ $\subseteq$ $A^\circ$.
\end{lemma}
\begin{theorem}\label{3.18}
    Let $A$ be a nonempty subset of metric space $X$. Then $A^\circ = \emptyset$ if and only if $Qrad_{X}(A) = 0$.
\end{theorem}

In the next theorem, using the above notions of quasi-center and quasi-radius, we determine the largest open balls contained in a subset $A$ of metric space $X.$ If $A$ is a proper subset of metric space $(X,d_X)$ such that $QCent_{X}(A)$ is empty, then there does not exist any open ball with largest radius that is contained in $A$. As $QCent_{X}(A)=\emptyset \implies Qrad_{X}(A)=\infty$. Let open ball $B(a,r)$ be the largest open ball entirely contained in $A$ for some $a\in A$ and $r>0$. This means $ B(a,r)\cap A^c=\emptyset \implies d_{X}(a,A^c)\geq r$. Now, as $Qrad_{X}(A)=\infty$ , for some $s>r, \exists$ $ b\in A$ such that $d_{X}(b,A^c)\geq s\implies B(b,s)$ contained in $A$, a contradiction. For example, $A=[0,\infty)\subseteq\mathbb{R}$ has no largest open ball contained in $A$.\par
For nonempty quasi-center of a subset $A$ of metric space $X$, we have the following result.

\begin{theorem}\label{3.19} Let  $A$ be a nonempty proper subset of metric space $X$. Then the largest open balls of X which are entirely contained  in A are the balls whose centers belong to $QCent_{X}(A)$ and radius is  equal to $Qrad_{X}(A)$.
\end{theorem}
    
     \begin{proof}
    Any point $a \in A$ is either an interior point of A or a boundary of A. 
    %If $A^\circ=\phi$ then, $A\subseteq\partial_{X}(A).$ And 
    If $a \in \partial_{X}(A)$ then  for every $\epsilon > 0$, the open ball $B(a,\epsilon)$
    %$ = \{y \in X| d_{X}(a,y)< \epsilon\}$
    intersects $A^{c}$. So, any ball centered at boundary point of A with positive radius can not be entirely contained in A. Thus if $A^\circ=\emptyset$ then $QCent_{X}(A)=A$ and the largest open balls contained in A are balls with zero radii.\par
    If $A^\circ\neq\emptyset$ then by Lemma \ref{3.17}, $QCent_{X}(A) \subseteq A^\circ.$ First, let $a \in A^\circ$ such that $a \notin QCent_{X}(A).$  As  $d_{X}(a,A^c)< Qrad_{X}(A),$ and $A$ is proper subset of $X$, then $\exists$ $b\in A^c$ such that $d_{X}(a,b)< Qrad_{X}(A).$ Thus  $B(a,Qrad_{X}(A)) \cap A^{c} \neq \emptyset.$ Therefore, any open ball centred at $a$ with radius $\geq Qrad_{X}(A)$ can not be entirely contained in A.
    Now, let $a \in QCent_{X}(A).$ Then $d_{X}(a,A^c) = Qrad_{X}(A).$ So, $B(a,Qrad_{X}(A))\subseteq A.$ Next, we observe that for $\epsilon>0$, open balls $B(a,Qrad_{X}(A)+\epsilon)$ has nonempty intersection with $A^{c}.$ As $\inf_{b\in A^c}d_{X}(a,b)=Qrad_{X}(A),$ we get that $\forall \epsilon>0$, $\exists$ $b'\in A^c$ such that $d_{X}(a,b')<Qrad_{X}(A)+\epsilon.$ So, $B(a,Qrad_{X}(A)+\epsilon)\cap A^c \neq \emptyset,$ and therefore any ball with radius greater than $Qrad_{X}(A)$ can not be entirely contained in $A.$ Hence, our claim.
\end{proof}

\begin{corollary}\label{3.22}
	Let $A$ and $B$ be proper subsets of a metric space $X$ such that $QCent_{X}(A)\neq\emptyset$ and $A\subseteq B.$ Then $Qrad_{X}(A)\leq Qrad_{X}(B)$.
\end{corollary}

\begin{proof}
	If $A^{\circ} = \emptyset$ then 
	the result follows from 
	Theorem \ref{3.18}. If $A^{\circ} \neq \emptyset$ then by Theorem \ref{3.19}, $B(a,Qrad_{X}(A)$ is the largest open ball contained in A, where $a\in QCent_{X}(A)$. As $B(a,Qrad_{X}(A))\subseteq A\subseteq B$ and $B(b,Qrad_{X}(B))$ is the largest ball contained in $B$, where $b\in QCent_{X}(B).$ Hence, our claim.
	%This implies that $Qrad_{X}(A)\leq Qrad_{X}(B).$ 
\end{proof}

But if $A \subseteq B$ are proper subsets of metric space $X$ then it does not imply that  $QCent_{X}(A) \subseteq QCent_{X}(B)$. For example: Take subsets  $A=\left[0,1\right]$ and $B=\left[0,2\right]$ of $\mathbb{R}$ with the usual metric. Then $QCent_{\mathbb{R}}(A)=\{\frac{1}{2}\}$ and $Qent_{\mathbb{R}}(B)=\{1\}$.

\begin{remark}\label{5.11}
    
 {\normalfont Note that if $A\subseteq B$ are  proper subsets of a path metric space $X$ such that $Cent_{X}(A)\neq\emptyset$, then by Corollary \ref{3.22}, $rad_{X}(A)\leq rad_{X}(B).$}
\end{remark}

 %Now, we see a relationship between quasi-radius and diameter of subset $A$ of a metric space $X$:

\begin{remark}
	{\normalfont We can notice that the radius of a subset may not be equal to half of its diameter. Infact, it is easy to observe that if $A$ is a nonclopen subset of metric space $X$ such that $Cent_{X}(A)\neq\emptyset$, then $rad_{X}(A)\leq diam_{X}(A)$. By Theorem \ref{3.19}, we also observe that for a  proper subset $A$ of the Euclidean space $\mathbb{R}^n$, having nonempty center,  $rad_{\mathbb{R}^n}(A)\leq \frac{1}{2} diam_{\mathbb{R}^n}(A)$. 
		%For $a_0\in A$, we have $\inf_{a\in \partial_{X}(A)}d_{X}(a_0,a)\leq \sup_{a\in \partial_{X}(A)}d_{X}(a_0,a)\leq \sup_{a\in \overline{A}}d_{X}(a_0,a)$. $\implies rad_{X}(A)=\sup_{a_0\in A}(\inf_{a\in \partial_{X}(A)}d_{X}(a_0,a))\leq \sup_{a_0\in \overline{A}}(\sup_{a\in \overline{A}}d_{X}(a_0,a))=diam_{X}(A)$.
	}
\end{remark}

%\begin{corollary}
    %Let $X$ be metric space and $A$ be a proper subset of X then $Qrad_{X}(A) \leq \frac{1}{2} diam_{X}(A)$.
   % Let $A$ be a non empty proper subset of Euclidean space $\mathbb{R}^n$ then $rad_{\mathbb{R}^n}(A)\leq diam_{\mathbb{R}^n}(A)$.
%\end{corollary}

   % \begin{proof}
      %  By Theorem \ref{3.19}, $B_{X}(a,rad_{X}(A)) \subseteq A,$ where $a \in Cent_{X}(A)$. Which means $diam_{X}(B(a,rad_{X}(A)))\leq diam_{X}(A)\implies 2rad_{X}(A)\leq diam_{X}(A).$ 
    %\end{proof}

%Above corollary also implies that for a proper bounded subset $A$ of a metric space $X,$ $Qrad_{X}(A)$ is finite. Also note that if $A\subseteq B$ are bounded proper subsets of a path metric space $X$ then $rad_{X}(A)\leq rad_{X}(B),$ and $rad_{X}(A) \leq \frac{1}{2} diam_{X}(A)$.\\

For subsets $A$ and $B$ of a metric space $X$, we observe that $diam_{X}(A) \leq diam_{X}(B)$ does not imply $rad_{X}(A) \leq rad_{X}(B)$ or $Qrad_{X}(A) \leq Qrad_{X}(B)$.
\begin{example}
	{\normalfont    Consider $ A_{1}\;\text{and}\; A_{2} \subseteq \mathbb{R}^{2}$, where $A_{1}$ is the line segment joining  $(-2,0)$ and $(2,0)$ in $\mathbb{R}^{2}$ and $A_{2}$ is the closed ball centered at $(4,0)$ with radius 1.
   % As $rad_{\mathbb{R}^{2}}(A_{1}) =0 <1 =rad_{\mathbb{R}^{2}}(A_{2})$, so by Corollary \ref{3.33}, we have $Cent_{\mathbb{R}^{2}}(A)= Cent_{\mathbb{R}^{2}}(A_{2}) =\{(4,0)\}$ and $rad_{\mathbb{R}^{2}}(A) = rad_{\mathbb{R}^{2}}(A_{2}) = 1$.
   Notice that $diam_{\mathbb{R}^{2}}(A_{1}) =4 > 2 =diam_{\mathbb{R}^{2}}(A_{2})$ but $Qrad_{\mathbb{R}^{2}}(A_{1})=rad_{\mathbb{R}^{2}}(A_{1}) =0 <1 =rad_{\mathbb{R}^{2}}(A_{2})=Qrad_{\mathbb{R}^{2}}(A_{2})$.}
	%Also notice that the center of a disconnented bounded subset is not necessarily equal to  the union of centers of the components with maximum diameter.
\end{example}
Next, we introduce a notion of concentric subsets.

\begin{definition}[Concentric Subsets]
	 Two subsets $A$ and $B$ of a metric space $X$ are called concentric subsets if they have the same nonempty centers in X.
\end{definition}

\begin{example}\label{3.20}
	{\normalfont    Take $A=\left[-2,2\right]$ and $B=\left[-1,1\right]$ in the set $\mathbb{R}$ of  real numbers with the usual metric. Here, $Cent_{\mathbb{R}}(A) =Cent_{\mathbb{R}}(B)= \{0\}$. So, A and B are concentric in $\mathbb{R}$. Infact,   all intervals of the form $[-n,n]$ or $(-m,m)$ are concentric in $\mathbb{R},$ where $m,n$ are positive real numbers.}
\end{example}

\begin{example}\label{3.21}
	{\normalfont    Let $\mathbb{R}^{2}$ be the real plane with the Euclidean metric. 
		Then the unit circle $\mathbb{S}^{1}$ and the punctured disc $A=\{(x,y) \in \mathbb{R}^{2} | x^{2}+y^{2}\leq 4\} \backslash \{(0,0)\}$ are concentric as $Cent_{\mathbb{R}^{2}}(\mathbb{S}^{1}) = Cent_{\mathbb{R}^{2}}(A) = \mathbb{S}^{1}$.}
\end{example}

\begin{remark}
	{\normalfont %\begin{itemize}
			%\item[(i)] Every subset A of a metric space is cocentric to itself.
			%\item[(ii)]  Two cocentric subsets may not have the same radius. In Example \ref{3.20}, A and B are cocentric but $rad_{\mathbb{R}}(A)=2$ and $rad_{\mathbb{R}}(B)=1$. In Example \ref{3.21}, $\mathbb{S}^{1}$ and $A$ are cocentric but $rad_{\mathbb{R}^{2}}(\mathbb{S}^{1})=0$ and $rad_{\mathbb{R}^{2}}(A)=1$.
			%\item[(iii)]  
			Concentric subsets may not be contained in each other. For example: Let $A=\mathbb{D}^{2} \cup \{(2,0)\}$ and $B = \mathbb{D}^{2} \cup \{(0,2)\}$ be two subsets of the Euclidean plane  $\mathbb{R}^{2},$ where $\mathbb{D}^{2}$ is the unit disc in $\mathbb{R}^{2}$. Note that  $Cent_{\mathbb{R}^{2}}(A) = Cent_{\mathbb{R}^{2}}(B) = \{(0,0)\}$, but neither $A\subseteq B$ nor $B\subseteq A.$ Also, notice that if $A$ and $B$ are concentric  subsets with the same radius then $A$ may not be equal to $B$.}
%	\end{itemize}}
	
\end{remark}
% If A and B are cocentric subsets of a metric space $X$ then $A \cap B \neq \phi$. 
It is easy to see that the relation of being concentric subsets is an equivalence relation on the class of subsets of $X$ with nonempty centers.
% \item [(3)] If A and B are cocentric and non clopen subsets of a path metric space $X$ with the same radius. Then, the largest open balls contained in A and B are the same balls.

\begin{theorem}\label{5.17}
     Let $A$ be a subset of a path metric space $X$ such that $A$ has nonempty interior and $Cent_{X}(A) \neq \emptyset$. Then $A$ and $A^\circ$ are concentric with same radii.
\end{theorem}
    
\begin{proof}
    By Theorem \ref{3.28} and Remark \ref{5.11}, we get $rad_{X}(A)= rad_{X}(A^\circ).$
    Let $x\in Cent_{X}(A).$ Then by Lemma \ref{B}, $x\in A^\circ.$ We have $rad_{X}(A)= rad_{X}(A^\circ)\geq d_{X}(x,\partial_{X}(A^\circ))\geq d_{X}(x,\partial_{X}(A))=rad_{X}(A).$ Thus $d_{X}(x,\partial_{X}(A^\circ))=rad_{X}(A^\circ).$ So, $ x\in Cent_{X}(A^\circ).$ Therefore, $Cent_{X}(A)\subseteq\ Cent_{X}(A^\circ)$. Now, let $x \notin Cent_{X}(A)=QCent_{X}(A).$ So, $ d_{X}(x,A^c)<rad_{X}(A)= Qrad_{X}(A).$ So, $B(x,rad_X(A))\cap A^c\neq\emptyset$ $\implies B(x,rad_X(A^\circ))\cap A^c\neq\emptyset$. This implies that $B(x,rad_X(A^\circ))\cap (A^\circ)^c\neq\emptyset. $ Thus $x\notin QCent_{X}(A^\circ) $ $= Cent_{X}(A^\circ).$ Therefore, $Cent_{X}(A^\circ)\subseteq Cent_{X}(A).$ Hence,  $Cent_{X}(A^\circ)= Cent_{X}(A).$
\end{proof}

\begin{remark}
      \normalfont  %Theorem \ref{3.28} and 
      Theorem \ref{5.17} may not be true if $A^\circ= \emptyset$. For example: Consider the set $\mathbb{R}$ of real numbers with the usual metric. If $A=\{\frac{1}{n} |n \in \mathbb{N}\}$ $\subseteq \mathbb{R},$  %$\overline{A}=\{0\} \cup \{\frac{1}{n} |n \in \mathbb{N}\}$ and
      then $A^\circ = \emptyset$. We have $Cent_{X}(A^\circ)= \emptyset$ $\&$ $Cent_{X}(A)= A$  %$Cent_{X}(\overline{A})= \overline{A}$ 
      and $rad_{X}(A^\circ) = \infty$ $\&$ $rad_{X}(A) = 0$. Notice that $rad_{X}(A^\circ)\neq rad_{X}(A),$ %and $Cent_{X}(\overline{A}) \not\subseteq Cent_{X}(A)$ 
      and $ Cent_{X}(A) \neq Cent_{X}(A^\circ)$.
\end{remark}

If $X$ is a path metric space in Theorem \ref{3.37}, then we have the following result.
%observe that the center of a disconnected proper bounded subset of $X$  is equals to the union of centers of its components with the maximum radius and its radius is equal to the radius of component with the maximum radius.

\begin{corollary}\label{3.38}
		Let $(X,d_X)$ be a path metric space. For nonempty proper subsets $A_i \subseteq X, 1\leq i\leq n,$ such that $A_i$ and $A_j$ are separated,
		%$\overline{A_i}\cap \overline{A_j}=\emptyset$
		 for all $i\neq j$ and $n\in \mathbb{N},$ and let
	% be a bounded non clopen subset of a metric space X such that $A= A_{1} \cup A_{2} \cup ...\cup A_{n},$ where $A_{i}'s$ are non clopen subsets of $X$ such that $\overline{A_i}\cap \overline{A_j}=\emptyset,\,\forall i,j$ and $n\in\mathbb{N}$.
	%let $\Tilde{A}_{j}=\{a\in Cent_{X}(A_{j})|d_{X}(a,\partial_{X}(A_{i}))<rad_{X}(A_{j}),\text{ for some } i\neq j\}, 1\leq j\leq n$, and
	 $M$ be the collection of all those $A_j$ such that $rad_{X}(A_{j})= \max\{rad_{X}(A_{i})|1\leq i \leq n\}$ and $Cent_{X}(A_j)\neq\emptyset$. Then, if  %\begin{itemize}
		%\item[(i)] 
	    $M\neq\emptyset$,
		%$\bigcup\limits_{A_j\in M}Cent_{X}(A_j)\neq\emptyset,$ 
		then $Cent_{X}(\bigcup\limits_{1}^{n}A_{i}) = \bigcup\limits_{A_{j}\in M}Cent_{X}(A_{j})$ $\&$ $rad_{X}(\bigcup\limits_{1}^{n}A_{i}) = \max\{rad_{X}(A_{i})|1\leq i \leq n\}.$ 
	%	\item[(ii)] if %$M=\emptyset$,
	%	$\bigcup\limits_{A_j\in M}Cent_{X}(A_j)=\emptyset$,
	%	then $rad_{X}(\bigcup\limits_{1}^{n}A_{i}) <  \max\{rad_{X}(A_{i})|1\leq i\leq n\}$.
		
	%\end{itemize}
	%Let $A \subseteq X$ be a non empty proper bounded subset of a path metric space $X$ such that $A= A_{1} \cup A_{2} \cup ...\cup A_{n},$ where $A_{i}'s$ are non empty proper subsets of $X$ such that $\overline{A_i}\cap\overline{A_j}=\phi,\forall\,i,j$ and $n\in \mathbb{N}.$ Let $K$ be the collection of all those $A_j$ such that $rad_{X}(A_{j})= max\{rad_{X}(A_{i})|1\leq i\leq n\}$. Then $Cent_{X}(A) = \cup_{A_{j}\in K}Cent_{X}(A_{j})$ and $rad_{X}(A) = max\{rad_{X}(A_{i})|1\leq i \leq n\}$.
\end{corollary}

\begin{proof}
	%As $A_j's$ are non clopen subsets of X, by Theorem \ref{3.6}, $Cent_{X}(A_j)\neq\phi,\forall\, j $. Next 
	We observe that $\Tilde{A_j}=\emptyset,\, \forall\, j,$ where $\Tilde{A_j}$ is the same as defined in Theorem \ref{3.37}. As $A_i$ and $A_j$ are separarted for all %$\overline{A_i}\cap\overline{A_j}=\emptyset,\forall\,
	$ i\neq j,$  $\partial_{X}(A_i)\subseteq (A_j)^c.$ So, we get $d_{X}(a,(A_j)^c)\leq d_{X}(a,\partial_{X}(A_i)), \forall \,a\in A_j.$ As $X$ is path metric space, by Remark \ref{5.5}, we get $rad_{X}(A_j)=Qrad_{X}(A_j)\leq d_{X}(a,\partial_{X}(A_i)),\forall\,a\in Cent_{X}(A_j).$ This implies that $\Tilde{A_j}=\emptyset,\forall\,j.$ Now, the result follows from Theorem \ref{3.37}.
\end{proof}

\begin{remark}
	{\normalfont If $X$ is a Euclidean space then the center of a disconnected proper subset of $X$  is equal to the union of centers of its connected components with the maximum radius and its radius is equal to the radius of component with the maximum radius.}
\end{remark}

\begin{example}
	{\normalfont  Let $A= [0,1], B= [2,6]$ and $C= [8,12]$ be subsets of $\mathbb{R}$ with the usual metric. Here $rad_{\mathbb{R}}(B)=rad_{\mathbb{R}}(C)= 2 > 0.5 =rad_{\mathbb{R}}(A)$. So, by corollary \ref{3.38}, $Cent_{\mathbb{R}}(A \cup B\cup C) = Cent_{\mathbb{R}}(B)\cup Cent_{\mathbb{R}}(C) = \{4,10\}$ and $rad_{\mathbb{R}}(A \cup B\cup C)= rad_{\mathbb{R}}(B)= 2.$}
\end{example}

	%Here we can not replace the condition of maximum radius of component by maximum diameter of component. Which means, we can't take $K$ to be the collection of all those $A_j$ such that $diam_{X}(A_{j})= max\{diam_{X}(A_{i})|A_{i}\in \{A_{1},A_{2},...A_{n}\}\}$.

%\begin{remark}
%{\normalfont   In a metric space $X$ for $A,B \subseteq X$, $diam_{X}(A) \leq diam_{X}(B) $ neither imply $rad_{X}(A) \leq rad_{X}(B)$ nor imply  $Qrad_{X}(A) \leq Qrad_{X}(B)$. As in the above example, $diam_{\mathbb{R}^{2}}(A_{1}) =4 > 2 =diam_{\mathbb{R}^{2}}(A_{2})$ but $Qrad_{\mathbb{R}^{2}}(A_{1})=rad_{\mathbb{R}^{2}}(A_{1}) =0 <1 =rad_{\mathbb{R}^{2}}(A_{2})=Qrad_{\mathbb{R}^{2}}(A_{2})$.}
%\end{remark}
 
\section{\textbf{Conjecture}}
%Let $X$ be a metric space and $A$ be a non clopen subset of $X$. Let $p:A\longrightarrow \mathbb{R}$ be such that $p(a)=d_{X}(a,\partial_{X}(A)),$ and for $\alpha\geq 0$, $P_{\alpha}=p^{-1}(-\infty,\alpha]=\{a\in A|p(a)\leq\alpha\},$ which means $P_\alpha$ consist of all those points of $A$ which arec atmost $\alpha$ distance away from $\partial_{X}(A)$. Notice that $P_0=\partial_{X}(A)\cap A$ and those point of $A$ which are at the maximum distance from $\partial_{X}(A)$ are $rad_{X}(A)$ distance away from $\partial_{X}(A)$, so for $\alpha\geq rad_{X}(A)$ we have $P_\alpha=A$.\par
%Now, consider the filtration $P_0\subseteq P_1\subseteq P_2\subseteq...P_{rad_{X}(A)}.$ In this filtration we start from $\partial_{X}(A)\cap A$ and after when $\alpha\geq rad_{X}(A)$ we get the whole $A$. So this filtration shows how the subset $A$ is evolved from its boundary. During the filtration some classes are created and destroyed.
Let $A$ be a nonclopen subset of metric space $X$. Consider a map $p:A\longrightarrow \mathbb{R}$ such that $p(a)=d_{X}(a,\partial_{X}(A))$. Then the sublevel sets $P_{\alpha}=p^{-1}(-\infty,\alpha]$ for $\alpha\geq 0$ form a 1-parameter family of nested subsets, $P_\alpha\subseteq P_\beta$ whenever $\alpha\leq \beta$. This gives a filtration $\partial_{X}(A)\cap A= P_{\alpha_0}\subseteq P_{\alpha_1}\subseteq P_{\alpha_2}\subseteq...P_{rad_{X}(A)}=A$, where $0=\alpha_0\leq \alpha_1\leq\alpha_2\leq...\leq rad_{X}(A).$
This family of sublevel sets describes how a set $A$ evolves from the set $\partial_{X}(A)\cap A$ as the threshold increases.
%Initialy in this filtration we have those boundary points of $A$ which belong to $A$ and at each sublevel we add some more points of $A$ to its previous sublevel set based on their distance from $\partial_{X}(A)$.
Notice that $Cent_{X}(A)\cap P_\alpha=\emptyset$ for every $\alpha<rad_{X}(A)$.\\
Let $A$ be a proper subset of the Euclidean space $\mathbb{R}^n$ such that $A$ has positive radius. Then for every $\delta>0$ and $\alpha=rad_{X}(A)-\delta$, we have $Cent_{X}(A)\cap P_\alpha=\emptyset$ and for every connected component $K$ of $Cent_{X}(A),$ the ball $B_\delta(K)=\{b\in A|d_{X}(b,K)<\delta\}\cap P_\alpha=\emptyset$, which create an $(n-1)$-dimensional cavity around $K$ in $P_\alpha$. 
%When $\alpha=rad_{X}(A)$ we have $Cent_{X}(A)\subseteq P_\alpha =A$, which means
As $\delta \rightarrow 0$ all these cavities around the connected components of $Cent_{X}(A)$ will vanish.\par 
Above observation leads us to a conjecture that describes a relation between the rank $\beta_{n-1}(A)$ of $(n-1)$-dimensional homology group of $A$ and the rank  $\beta_{0}(Cent_{X}(A))$ of 0-th homology group of $Cent_{X}(A)$.
\begin{conjecture}
	Let $A$ be a proper subset of the Euclidean space $\mathbb{R}^n$ with positive radius. Then there exists a sublevel set $P_\alpha$ with $\alpha<rad_{X}(A),$ %$rad_{X}(A)-\delta$ for $\delta>0$ 
	such that $\beta_{n-1}(P_\alpha)=\beta_{n-1}(A)+\beta_{0}(Cent_{X}(A))$.
	% where $\epsilon=rad_{X}(A)-\delta$ for some $\delta>0$ and $\beta_{0}(Cent_{X}(A))$ is the number of connected components in $Cent_{X}(A).$
\end{conjecture}
From the above observation we can also conclude that at least $\beta_{0}(Cent_{X}(A))$ number of classes disappear or die in the $(n-1)$-th persistence diagram or $(n-1)$ dimensional barcode \cite{Edelsbrunner} of $A$ made with above filtration, when $\alpha=rad_{X}(A)$.\\

\hspace{50mm} \textbf{Acknowledgement}

\vspace{2mm}
  We would like to thank Omer Cantor from the Department of Mathematics, University of Haifa, Israel,  for their valuable comments and suggestions which have helped
us to improve the original version of the paper considerably.

\bibliographystyle{plain}

\end{document}